%% file: KaehlerDiffAlg.tex
\documentclass[draft]{amsart}

\usepackage{amsmath, amssymb, amsthm, amscd}
\usepackage{amsfonts, mathrsfs, amsbsy}
\theoremstyle{plain}

\newtheorem{thm}{Theorem}[section]
\newtheorem{prop}[thm]{Proposition}
\newtheorem{lem}[thm]{Lemma}
\newtheorem{cor}[thm]{Corollary}
\theoremstyle{definition}
\newtheorem{defn}[thm]{Definition}
\newtheorem{exam}[thm]{Example}

\newtheorem{conjecture}[thm]{Conjecture}
\newtheorem{rem}[thm]{Remark}
\theoremstyle{remark}

\numberwithin{equation}{subsection}
\newtheorem{ques}[thm]{Question}

       \def\im{\operatorname{Im}}
\def\Ker{\operatorname{Ker}}       
     
\def\deg{\operatorname{deg}}

       \def\Supp{\operatorname{Supp}}

       \def\ri{\operatorname{ri}}
\def\HF{\operatorname{HF}}         \def\HP{\operatorname{HP}}

\newcommand{\bbQ}{\ensuremath{\mathbb Q}}
\newcommand{\bbZ}{\ensuremath{\mathbb Z}}
\newcommand{\bbN}{\ensuremath{\mathbb N}}
\newcommand{\bbP}{\ensuremath{\mathbb P}}
\newcommand{\bbW}{{\ensuremath{\mathbb W}}}
\newcommand{\bbV}{\ensuremath{\mathbb V}}
\newcommand{\bbX}{\ensuremath{\mathbb X}}
\newcommand{\bbY}{\ensuremath{\mathbb Y}}

\newcommand{\calJ}{{\mathcal{J}}}

\newcommand{\calZ}{{\mathcal{Z}}}

\begin{document}

\title[K\"{a}hler Differential Algebras for 0-dimensional Schemes]
{K\"{a}hler Differential Algebras for 0-dimensional Schemes}
\author{Martin Kreuzer}
\address[Martin Kreuzer]{
Fakult\"{a}t f\"{u}r Informatik und Mathematik \\
Universit\"{a}t Passau, D-94030 Passau, Germany}
\email{martin.kreuzer@uni-passau.de}

\author{Tran N. K. Linh}
\address[Tran N. K. Linh]{Department of Mathematics,
Hue University's College of Education,
34 Le Loi, Hue, Vietnam}
\email{tnkhanhlinh141@gmail.com}

\author{Le Ngoc Long}
\address[Le Ngoc Long]{
Fakult\"{a}t f\"{u}r Informatik und Mathematik \\
Universit\"{a}t Passau, D-94030 Passau, Germany \newline
\hspace*{.5cm} \textrm{and} Department of Mathematics,
Hue University's College of Education,
34 Le Loi, Hue, Vietnam}
\email{nglong16633@gmail.com}



\date{\today}

\dedicatory{}

\commby{Martin Kreuzer, Tran N.K. Linh and Le Ngoc Long}


\begin{abstract}
Given a 0-dimensional scheme in a projective $n$-space~$\bbP^n$
over a field~$K$, we study the K\"ahler differential algebra
$\Omega_{R_\bbX/K}$ of its homogeneous coordinate ring $R_\bbX$.
Using explicit presentations of the modules $\Omega^m_{R_\bbX/K}$
of K\"{a}hler differential $m$-forms, we determine many values
of their Hilbert functions explicitly and bound their Hilbert
polynomials and regularity indices.
Detailed results are obtained for subschemes of~$\bbP^1$,
fat point schemes, and subschemes of~$\bbP^2$ supported on
a conic.
\end{abstract}

\maketitle

\section{Introduction}

In the paper \cite{DK}, G. de Dominicis and the first author
introduced the application of K\"{a}hler differential modules
to the study of 0-dimensional subschemes $\bbX$ of a projective
space $\bbP^n$ over a field $K$ of characteristic zero.
They showed that this graded module over the homogeneous
coordinate ring $R_\bbX$ contains numerical and algebraic
information which is not readily available from
the homogeneous vanishing ideal or from $R_\bbX$.
Later, in \cite{KLL}, the authors extended and refined these
techniques for fat point schemes in $\bbP^n$.
Following the classical construction described by E. Kunz
in his book \cite{Ku2}, it is natural to define the
K\"ahler differential algebra 
$\Omega_{R_\bbX/K}=\bigoplus_{m\in\bbN}\Omega^m_{R_{\bbX}/K}$
of~$\bbX$ as the exterior algebra of its K\"ahler 
differential module~$\Omega^1_{R_\bbX/K}$.
This invites the question whether the K\"ahler differential
algebra contains numerical and algebraic information
about $\bbX$ which is not readily available in~$R_\bbX$
or in~$\Omega^1_{R_\bbX/K}$.
Thus the following example provided the initial spark
to ignite the curiosity of the authors.

\begin{exam}[See Example~\ref{S1.Exam.06}]
Let $\bbX$ and $\bbY$ be two sets of six reduced $K$-rational
points in $\bbP^2$ such that $\bbX$ is contained in
a non-singular conic, and such that $\bbY$ consists of
three points on a line and three points on another line.
Then the Hilbert functions of~$R_\bbX$ and $R_\bbY$ agree,
as do the Hilbert functions of~$\Omega^1_{R_\bbX/K}$
and $\Omega^1_{R_\bbY/K}$.
However, the Hilbert functions of~$\Omega^2_{R_\bbX/K}$
and $\Omega^2_{R_\bbY/K}$ are different,
and also the Hilbert functions of~$\Omega^3_{R_\bbX/K}$
and $\Omega^3_{R_\bbY/K}$ disagree:
\[
\begin{array}{ll}
  \HF_{\Omega^2_{R_\bbX/K}}:
  \ 0\ 0\ 3\ 6\ {\bf 4}\ 1\ 0\ 0\cdots,
  & \HF_{\Omega^2_{R_\bbY/K}}:
  \ 0\ 0\ 3\ 6\ {\bf 5}\ 1\ 0\ 0\cdots,\\
  \HF_{\Omega^3_{R_\bbX/K}}:
  \ 0\ 0\ 0\ 1\ {\bf 0}\ 0\ \cdots,
  & \HF_{\Omega^3_{R_\bbY/K}}:
  \ 0\ 0\ 0\ 1\ {\bf 1}\ 0\ \cdots.
\end{array}
\]
So, the Hilbert functions of the exterior powers
of~$\Omega^1_{R_\bbX/K}$ ``know'' whether $\bbX$
is contained in an irreducible or a reducible conic.
\end{exam}
This observation motivated the studies underlying
this paper. Let us now outline its contents in more detail.
In the second section we start by recalling the definitions
of the K\"ahler differential module $\Omega^1_{R_\bbX/K}$
and the K\"ahler differential algebra 
$\Omega_{R_\bbX/K}=\bigwedge_{R_\bbX}(\Omega^1_{R_{\bbX}/K})$
of a 0-dimensional subscheme $\bbX$ of~$\bbP^n$.
As explained in \cite{Ku2}, we can calculate an explicit
presentation of $\Omega^m_{R_\bbX/K}=
\bigwedge^{m}_{R_\bbX}(\Omega^1_{R_{\bbX}/K})$ 
for every $m\ge 1$.
Moreover, we show that $\Omega^m_{R_\bbX/K}=\langle 0\rangle$
for $m>n+1$, provide a simplified presentation for
$\Omega^{n+1}_{R_\bbX/K}$, and show that the Koszul complex
yields an exact sequence
\[
0 \longrightarrow \Omega^{n+1}_{R_{\bbX}/K}
\longrightarrow \Omega^n_{R_{\bbX}/K}
\longrightarrow \cdots \longrightarrow  \Omega^2_{R_{\bbX}/K}
\longrightarrow \Omega^1_{R_{\bbX}/K}
\longrightarrow \mathfrak{m}_{\bbX} \longrightarrow 0
\]
where $\mathfrak{m}_{\bbX}$ is the homogeneous maximal
ideal $\mathfrak{m}_{\bbX}=\langle x_0,\dots,x_n\rangle$
of~$R_\bbX$.

In Section 3 we have a brief glance at the case $n=1$,
i.e., at 0-dimensional subschemes of a projective line.
Unsurprisingly, in this case the Hilbert functions
and regularity indices of~$\Omega^1_{R_{\bbX}/K}$
and $\Omega^2_{R_{\bbX}/K}$ can be written down explicitly.

In Section 4 we look at the Hilbert function of
$\Omega^m_{R_{\bbX}/K}$ in special degrees. We provide
explicit values in low degrees, show that the Hilbert
polynomial (i.e., the value in high degrees) is constant,
and examine monotonicity in intermediate degrees.
These insights are accompanied by a bound for the
regularity index of $\Omega^m_{R_{\bbX}/K}$ in terms
of the regularity index of $\Omega^1_{R_{\bbX}/K}$
in Section 5.

Then, in the next four sections, we look at the modules
of K\"{a}hler differential $m$-forms for fat point
schemes~$\bbW$. Such schemes are defined by ideals
of the form $I_\bbW =
\wp_1^{m_1}\cap\cdots\cap \wp_s^{m_s}$,
where the ideals $\wp_i$ are the vanishing ideals
of distinct $K$-rational points in~$\bbP^n$.
In Section~6 we prove a regularity bound for
$\Omega^m_{R_\bbW/K}$ which uses the regularity index
of the fattening of $\bbW$, i.e., the scheme
$\bbV$ defined by
$I_{\bbV} = \wp_1^{m_1+1}\cap\cdots\cap \wp_s^{m_s+1}$.
For fat point schemes, we also give bounds on some specific
values of the Hilbert polynomial of~$\Omega^m_{R_\bbW/K}$.
In the reduced case (i.e., when $m_1=\cdots=m_s=1$), these values
are zero for $m\ge 2$, but as soon as one of the exponents $m_i$
satisfies $m_i\ge 2$, not all of these values are zero anymore.
Thus the property of $\bbW$ to be reduced is reflected in
the values of the Hilbert polynomials of~$\Omega^m_{R_\bbW/K}$
(see Cor.~\ref{S6.Cor.04}).
More generally, Prop.~\ref{S6.Prop.03} provides upper
and lower bounds for these Hilbert polynomials.

For the highest non-zero module of K\"{a}hler differentials
$\Omega^{n+1}_{R_{\mathbb{W}}/K}$,  we can sometimes determine
its Hilbert polynomial explicitly. More precisely, we have formulas
for schemes $\mathbb{W}$ contained in a hyperplane
(see Prop.~\ref{S7.Prop.01}) and for equimultiple
schemes~$\bbW$ (see Thm.~\ref{S7.Thm.03}).
Another case in which we have more detailed information
is the module $\Omega^{2}_{R_\bbW/K}$ for an
equimultiple fat point scheme~$\bbW$
(i.e., a scheme satisfying $m_1=\dots=m_s$).
In this case we can extract the value of the Hilbert polynomial
of $\Omega^{2}_{R_\bbW/K}$ from a complex connecting
it to its fattening and second fattening
(see Prop.~\ref{S7.Prop.06}).
We end the discussion of Hilbert polynomials with a conjecture
for their value for $\Omega^{n+1}_{R_\bbW/K}$.

In Section~9, a rich and detailed set of results describes
the Hilbert functions of $\Omega^{m}_{R_\bbW/K}$,
where $m=1,2,3$, in the case of a fat point scheme~$\bbW$
in~$\bbP^2$ supported on a non-singular conic.
In this case, the Hilbert function of $\Omega^1_{R_\bbW/K}$
can be computed explicitly from the Hilbert functions
of suitable fat point schemes (see Thm.~\ref{S8.Thm.01}).
If $\bbW$ is an equimultiple fat point scheme, we construct
a special homogeneous system of generators of $I_\bbW$
in Prop.~\ref{S8.Lem.04} and use it to compute the Hilbert
function of~$\Omega^3_{R_\bbW/K}$ explicitly
(see Thm.~\ref{S8.Thm.06} and Prop.~\ref{S8.Prop.07}).
Consequently, we can use the exact sequence given by the
Koszul complex and determine the Hilbert function
of~$\Omega^2_{R_\bbW/K}$ explicitly
(see Prop.~\ref{S8.Prop.09}).

Finally, in the last section we point out the relation
between the K\"{a}hler differential algebra $\Omega_{R_\bbX/K}$
and the relative K\"{a}hler differential algebra
$\Omega_{R_\bbX/K[x_0]}$ and use it to deduce many properties
of the Hilbert function, the Hilbert polynomial, and the
regularity index of~$\Omega_{R_\bbX/K[x_0]}$
(see Propositions~\ref{S9.Prop.01}, \ref{S9.Prop.02}
and \ref{S9.Prop.03}).

Throughout the paper we illustrate all results with explicitly
computed examples. The necessary calculations were performed
using the second author's package for the computer algebra
system ApCoCoA (see \cite{ApC}). Unless explicitly stated
otherwise, we adhere to the definitions and notation
introduced in~\cite{KR1,KR2} and~\cite{Ku2}.

\medskip\bigbreak
\input{Sec1}

\medskip\bigbreak
\input{Sec2}

\medskip\bigbreak
\input{Sec3}

\medskip\bigbreak
\input{Sec4}

\medskip\bigbreak
\input{Sec5}

\medskip\bigbreak
\input{Sec6}

\medskip\bigbreak
\input{Sec7}

\medskip\bigbreak
\input{Sec8}

\medskip\bigbreak
\input{Sec9}

\medskip
\subsection*{Acknowledgment}
This paper is partially based on
the second author's dissertation~\cite{Lin}.
The authors are grateful to Lorenzo Robbiano
for his encouragement to elaborate
some of the results presented here.
The first author thanks Hue University's College
of Education (Vietnam) for the hospitality he experienced
during part of the preparation of this paper.
The second author would also like to acknowledge
her financial support from OeAD.

\bigbreak

\end{document}

%% file: Sec1.tex
\section{Definition and Basic Properties}

Throughout this paper we work over a field~$K$ of characteristic
zero. By $\bbP^n$ we denote the projective $n$-space over $K$.
The homogeneous coordinate ring of~$\bbP^n$ is $S=K[X_0,\dots,X_n]$.
It is equipped with the standard grading $\deg(X_i)=1$
for $i=0,\dots,n$.
Let $\bbX$ be a 0-dimensional scheme in $\bbP^n$, and
let $I_{\bbX}$ be the (saturated) homogeneous vanishing
ideal of~$\bbX$. Then the homogeneous coordinate ring of~$\bbX$
is $R_{\bbX}=S/I_{\bbX}$.
The ring $R_{\bbX} = \bigoplus_{i\ge 0} (R_{\bbX})_i$ is a
standard graded $K$-algebra. Its enveloping algebra is
$R_{\bbX}\otimes_K R_{\bbX} = \bigoplus_{i\ge 0}
(\bigoplus_{j+k=i} (R_{\bbX})_j\otimes (R_{\bbX})_k)$.
By $\calJ$ we denote the kernel of the homogeneous
$R_{\bbX}$-linear map of degree zero
$\mu:\ R_{\bbX}\otimes_K R_{\bbX}\rightarrow R_{\bbX}$
given by $\mu(f\otimes g)=fg$.
It is well known that $\mathcal{J}$ is the homogeneous ideal
of~$R_{\bbX}\otimes_K R_{\bbX}$ generated
by $\{x_i\otimes 1- 1\otimes x_i \mid 0\le i\le n\}$,
where $x_i$ is the image of $X_i$
in~$R_{\bbX}$ for~$i=0,\dots,n$.
In this paper we are interested in looking at the
algebraic structure and Hilbert function of the
following objects.

\begin{defn}
\begin{enumerate}
  \item The graded $R_{\bbX}$-module
    $\Omega^1_{R_{\bbX}/K} \,=\,\calJ/\calJ^2$ is called
    the {\bf module of K\"{a}hler differential 1-forms}
    of~$R_{\bbX}/K$, or simply the
    {\bf module of K\"{a}hler differentials}.

  \item The homogeneous $K$-linear map
    $d: R_{\bbX} \rightarrow \Omega^1_{R_{\bbX}/K}$ given by
    $f\mapsto f\otimes 1-1\otimes f+\calJ^2$ is
    called the {\bf universal derivation} of~$R_{\bbX}/K$.

  \item The $m$-th exterior power of $\Omega^1_{R_{\bbX}/K}$
  over $R_{\bbX}$ is called the {\bf module of
    K{\"a}hler differential $\mathbf{m}$-forms} of~$R_{\bbX}/K$
    and is denoted by~$\Omega^m_{R_{\bbX}/K}$.

  \item The direct sum
  $\Omega_{R_{\bbX}/K}:=\bigoplus_{m\in \bbN}\Omega^m_{R_{\bbX}/K}$
  is an $R_{\bbX}$-algebra. It is called the
  {\bf K{\"a}hler differential algebra} of~$R_{\bbX}/K$.
  Here we use $\Omega^0_{R_{\bbX}/K}=R_{\bbX}$.
\end{enumerate}
\end{defn}

More generally, for any graded $K$-algebra $T/S$, we can define
the module of K\"{a}hler differential $m$-forms
$\Omega^m_{T/S}$ and the K{\"a}hler differential algebra
$\Omega_{T/S}$ in analogously (cf. \cite[Section~2]{Ku2}).
The K{\"a}hler differential algebra of $\Omega_{R_{\bbX}/K}$
is in fact a bigraded $K$-algebra whose homogeneous component
in degree $(m,d)$ is given by $(\Omega^m_{R_{\bbX}/K})_d$.
Notice that we have $\deg(dx_i)=\deg(x_i)=1$ for
$i=0,\dots,n$. For $m\ge 0$, the graded $R_\bbX$-module
$\Omega^m_{R_{\bbX}/K}$ is finitely generated and
its Hilbert function is defined by
$$
\HF_{\Omega^m_{R_{\bbX}/K}}(i) = \dim_K (\Omega^m_{R_{\bbX}/K})_i
\,\quad\, \mbox{for all $i\in \mathbb{Z}$.}
$$
Note that $\Omega^0_{R_{\bbX}/K}=R_\bbX$
and $\Omega^1_{R_{\bbX}/K} = R_{\bbX}dx_0+\cdots+R_{\bbX}dx_n$.
Hence we obtain $\Omega^m_{R_{\bbX}/K} = \langle 0\rangle$
for $m>n+1$ and
$\Omega_{R_{\bbX}/K}=\bigoplus_{m=0}^{n+1}\Omega^m_{R_{\bbX}/K}$.
Furthermore, there is a presentation of $\Omega_{R_{\bbX}/K}$ as
$\Omega_{R_{\bbX}/K}\cong\Omega_{S/K}/
\langle I_{\bbX},dI_{\bbX}\rangle\Omega_{S/K}$
(cf. \cite[Proposition~4.12]{Ku2}).
From this we deduce the following presentation
of the module of K{\"a}hler differential $m$-forms.

\begin{prop}\label{S1.Prop.01}
Let $m\ge 1$ and let $\{F_1,\dots,F_r\}$ be a homogeneous system
of generators of~$I_\bbX$. The graded $R_\bbX$-module
$\Omega^m_{R_{\bbX}/K}$ has a presenation
$$
\Omega^m_{R_{\bbX}/K} \cong \Omega^m_{S/K}/(I_{\bbX}\Omega^m_{S/K}
+ dI_{\bbX}\Omega^{m-1}_{S/K})
$$
where $I_{\bbX}\Omega^m_{S/K}+ dI_{\bbX}\Omega^{m-1}_{S/K}$
is generated by
$$
\begin{aligned}
&\left\{
F_jdX_{i_1}\wedge\cdots\wedge dX_{i_m} \mid
1\le j\le r, 0\le i_1<\cdots<i_m\le n
\right\} \\
&\qquad\qquad\cup
\left\{
dF_j\wedge dX_{j_1}\wedge\cdots\wedge dX_{j_{m-1}} \mid
1\le j\le r, 0\le j_1<\cdots<j_{m-1}\le n
\right\}.
\end{aligned}
$$
\end{prop}

In the case $m=n+1$, the presentation of $\Omega^{n+1}_{R_{\bbX}/K}$
can be described explicitly as follows.

\begin{cor}\label{S1.Cor.02}
Let $\{F_1,\dots,F_r\}$ be a homogeneous system of generators
of~$I_\bbX$. There is an isomorphism of graded $R_{\bbX}$-modules
$$
\Omega^{n+1}_{R_{\bbX}/K} \cong
\big( S/
\big\langle\,
\tfrac{\partial F_j}{\partial X_i}
\mid  0\le  i \le n, 1\le j\le r
\,\big\rangle
\big)(-n-1).
$$
\end{cor}
\begin{proof}
Note that $\Omega^{n+1}_{S/K}$ is a free $S$-module of rank $1$
with basis $\{dX_0\wedge\cdots\wedge dX_n\}$, and so
$\Omega^{n+1}_{S/K} \cong S(-n-1)$.
For $F\in I_\bbX$ and $G\in S$, we have
$FdG = d(FG) - GdF \in dI_{\bbX}$.
It follows that
$I_{\bbX}\Omega^{m}_{S/K} \subseteq dI_{\bbX}\Omega^{m-1}_{S/K}$
for all $m\ge 1$.
Let $I =\langle\, \frac{\partial F_j}{\partial X_i} \mid
0\le  i \le n, 1\le j\le r \,\rangle$.
We need to show that
$dI_{\bbX}\Omega^{n}_{S/K} = IdX_0\wedge\cdots\wedge dX_n$.
Clearly, we have
$$
\tfrac{\partial F_j}{\partial X_i} dX_0\wedge\cdots\wedge dX_n
= (-1)^idF_j\wedge dX_0\wedge\cdots\wedge \widehat{dX_{i}}\wedge
\cdots\wedge dX_n \in dI_{\bbX}\Omega^{n}_{S/K}
$$
where $\widehat{dX_{i}}$ indicates that $dX_{i}$
is omitted in the wedge product.
Hence we get the inclusion
$dI_{\bbX}\Omega^{n}_{S/K} \supseteq IdX_0\wedge\cdots\wedge dX_n$.
For the other inclusion, let $F\in I_\bbX$, and let
$\{i_1,\dots, i_n\}\subseteq \{0,\dots, n\}$.
Write $F=G_1F_1+\cdots+G_rF_r$ with $G_1,\dots, G_r\in S$.
Then we have
$$
\begin{aligned}
dF-(F_1dG_1+\cdots +F_rdG_r) &=G_1dF_1+\cdots+G_rdF_r \\
&= G_1{\textstyle \sum\limits_{i=0}^n}
\tfrac{\partial F_1}{\partial X_i}dX_i +\cdots +
G_r{\textstyle \sum\limits_{i=0}^n}
\tfrac{\partial F_r}{\partial X_i}dX_i
\end{aligned}
$$
and hence
$(dF-(F_1dG_1+\cdots +F_rdG_r))\wedge
dX_{i_1}\wedge\cdots\wedge dX_{i_n}
\in IdX_0\wedge\cdots\wedge dX_n$.
Since the field $K$ has characteristic zero,
for $j=1,\dots,r$, Euler's relation yields that
$F_j= \frac{1}{\deg(F_j)}
\sum_{i=0}^nX_i\tfrac{\partial F_j}{\partial X_i}
\in I$.
In particular, we have $I_{\bbX} \subseteq I$.
This implies
$dF\wedge dX_{i_1}\wedge\cdots\wedge dX_{i_n}
\in IdX_0\wedge\cdots\wedge dX_n$.
Therefore we get the equality
$dI_{\bbX}\Omega^{n}_{S/K} = IdX_0\wedge\cdots\wedge dX_n$,
and the claim follows readily.
\end{proof}

Now let $\mathfrak{m}_{\bbX} = \langle x_0,\dots,x_n\rangle$
be the homogeneous maximal ideal of~$R_{\bbX}$, and let
$e: R_{\bbX}\rightarrow R_{\bbX}$ be
the {\bf Euler derivation} of~$R_{\bbX}/K$ given by
$f\mapsto i\cdot f$ for $f\in (R_{\bbX})_i$.
By universal property of~$\Omega^1_{R_{\bbX}/K}$
(cf.~\cite[Section~1]{Ku2}),
there is a unique homogeneous $R_{\bbX}$-linear map
$\gamma: \Omega^1_{R_{\bbX}/K} \rightarrow R_{\bbX}$
such that $e = \gamma\circ d$.
In particular, we have $\gamma(dx_i)=x_i$
for all $i=0,\dots,n$ and $\gamma(df) = \deg(f)\cdot f$
for every homogeneous element $f\in R_\bbX\setminus\{0\}$.
The Koszul complex of~$\gamma$ is the complex
$$
\cdots \stackrel{\gamma}{\longrightarrow}  \Omega^2_{R_{\bbX}/K}
\stackrel{\gamma}{\longrightarrow} \Omega^1_{R_{\bbX}/K}
\stackrel{\gamma}{\longrightarrow} \mathfrak{m}_{\bbX} \longrightarrow 0
$$
where $\gamma: \Omega^{m}_{R_{\bbX}/K} \rightarrow
\Omega^{m-1}_{R_{\bbX}/K}$ is a homogeneous $R_{\bbX}$-linear
map defined by
$$
\gamma(\omega_1\wedge\cdots\wedge\omega_m) =
{\textstyle\sum\limits_{j=1}^m} (-1)^{j+1}\gamma(\omega_j)
\cdot\omega_1\wedge\cdots\wedge
\widehat{\omega_j}\wedge\cdots\wedge \omega_m
$$
for all $\omega_1,\dots,\omega_m\in \Omega^1_{R_{\bbX}/K}$,
and where $\gamma(\omega\wedge \omega') =
\gamma(\omega)\wedge \omega' + (-1)^m\omega\wedge\gamma(\omega')$
for $\omega\in \Omega^{m}_{R_{\bbX}/K}$ and
$\omega'\in \Omega^{k}_{R_{\bbX}/K}$
(cf. \cite[1.6.1-2]{BH}). In our setting, this complex is
an exact sequence, as the following proposition shows.

\begin{prop}\label{S1.Prop.02}
The Koszul complex
\begin{equation}\tag{$\mathcal{K}$}
0 \longrightarrow \Omega^{n+1}_{R_{\bbX}/K}
\stackrel{\gamma}{\longrightarrow} \Omega^n_{R_{\bbX}/K}
\longrightarrow \cdots \longrightarrow  \Omega^2_{R_{\bbX}/K}
\stackrel{\gamma}{\longrightarrow} \Omega^1_{R_{\bbX}/K}
\stackrel{\gamma}{\longrightarrow} \mathfrak{m}_{\bbX}
\longrightarrow 0
\end{equation}
is an exact sequence of graded $R_{\bbX}$-modules.
\end{prop}

\begin{proof}
Let $1\le m\le n+1$, let $i\ge 0$, and
let $\omega = fdx_{i_1}\wedge\cdots\wedge dx_{i_m}
\in \Omega^{m}_{R_{\bbX}/K}$
with $f\in (R_\bbX)_i$ and $0\le i_1<\cdots<i_m\le n$.
Then we have
$$
(\gamma\circ d)(\omega) = ifdx_{i_1}\wedge\cdots\wedge dx_{i_m} -
df\wedge\gamma(dx_{i_1}\wedge\cdots\wedge dx_{i_m})
$$
and
$$
\begin{aligned}
(d\circ\gamma)(\omega) &=
d(f{\textstyle\sum\limits_{j=1}^m} (-1)^{j+1}x_{i_j}
dx_{i_1}\wedge\cdots\wedge \widehat{dx_{i_j}}
\wedge\cdots \wedge dx_{i_m})\\
&= df\wedge\gamma(dx_{i_1}\wedge\cdots\wedge dx_{i_m})
+ mfdx_{i_1}\wedge\cdots\wedge dx_{i_m}.
\end{aligned}
$$
This implies $(\gamma\circ d + d\circ\gamma)(\omega) = (m+i)\omega$.
Hence $(\gamma\circ d + d\circ\gamma)(\omega) = \deg(\omega)\omega$
for every homogeneous element $\omega \in \Omega^{m}_{R_{\bbX}/K}$.
Now suppose that $\omega \in \Omega^{m}_{R_{\bbX}/K}\setminus\{0\}$
is a homogeneous element with $\gamma(\omega)=0$.
Set $\widetilde{\omega}= \frac{1}{\deg(\omega)} d\omega
\in \Omega^{m+1}_{R_{\bbX}/K}$.
We get $\gamma(\widetilde{\omega})=\omega$,
and the proof is complete.
\end{proof}

Obviously, the ring $R_{\bbX}$ is Noetherian and
the graded $R_{\bbX}$-module $\Omega^m_{R_{\bbX}/K}$
is finitely generated, and so the {\bf Hilbert polynomial}
of~$\Omega^m_{R_{\bbX}/K}$ exists (cf.~\cite[5.1.21]{KR2})
and is denoted by $\HP_{\Omega^m_{R_{\bbX}/K}}(z)$.
The number
$\ri(\Omega^m_{R_{\bbX}/K}) = \min\{i\in \bbZ \mid
\HF_{\Omega^m_{R_{\bbX}/K}}(j)=\HP_{\Omega^m_{R_{\bbX}/K}}(j)
\mbox{\ for all $j\ge i$}\}$
is called the {\bf regularity index} of~$\Omega^m_{R_{\bbX}/K}$.
In the following, we denote the Hilbert function of~$R_\bbX$
by~$\HF_\bbX$ and its regularity index by~$r_{\bbX}$.
As a consequence of the exact sequence ($\mathcal{K}$), we have
the following bound for $\ri(\Omega^{n+1}_{R_{\bbX}/K})$.

\begin{cor}\label{S1.Cor.03}
We have
$\ri(\Omega^{n+1}_{R_{\bbX}/K})\le
\max\{r_{\bbX},\ri(\Omega^1_{R_{\bbX}/K}),
\dots,\ri(\Omega^n_{R_{\bbX}/K})\}$.
\end{cor}

Let us examine the Hilbert functions of the
modules of K{\"a}hler differential $m$-forms
and their regularity indices in a concrete case.

\begin{exam}\label{S1.Exam.06}
Let $\bbX$ and $\bbY$ be two sets of six reduced $K$-rational
points in $\bbP^2$ such that $\bbX$ is contained in
a non-singular conic and $\bbY$ lies on the union of two lines
and no 5 points of $\bbY$ are collinear.
Then the Hilbert functions of~$\bbX$ and $\bbY$ agree,
as do the Hilbert functions of~$\Omega^1_{R_\bbX/K}$
and $\Omega^1_{R_\bbY/K}$, namely
\[
\begin{array}{ll}
  \HF_\bbX = \HF_\bbY\!\!\! &:\ 1\ 3\ 5\ 6\ 6\ \cdots \\
  \HF_{\Omega^1_{R_\bbX/K}}=\HF_{\Omega^1_{R_\bbX/K}}\!\!\! &
  :\ 0\ 3\ 8\ 11\ 10\ 7\ 6\ 6\cdots.
\end{array}
\]
It is clear that $r_\bbX=r_\bbY =3$ and
$\ri(\Omega^1_{R_{\bbX}/K})=\ri(\Omega^1_{R_{\bbY}/K})=6$.
We also have $\HF_{\Omega^m_{R_\bbX/K}}(i)=
\HF_{\Omega^m_{R_\bbY/K}}(i)=0$
for $m=1,2,3$ and $i\le 0$.

However, the Hilbert functions of~$\Omega^2_{R_\bbX/K}$
and $\Omega^2_{R_\bbY/K}$ are different,
and also the Hilbert functions of~$\Omega^3_{R_\bbX/K}$
and $\Omega^3_{R_\bbY/K}$ disagree:
\[
\begin{array}{ll}
  \HF_{\Omega^2_{R_\bbX/K}}:
  \ 0\ 0\ 3\ 6\ {\bf 4}\ 1\ 0\ 0\cdots,
  & \HF_{\Omega^2_{R_\bbY/K}}:
  \ 0\ 0\ 3\ 6\ {\bf 5}\ 1\ 0\ 0\cdots,\\
  \HF_{\Omega^3_{R_\bbX/K}}:
  \ 0\ 0\ 0\ 1\ {\bf 0}\ 0\ \cdots,
  & \HF_{\Omega^3_{R_\bbY/K}}:
  \ 0\ 0\ 0\ 1\ {\bf 1}\ 0\ \cdots.
\end{array}
\]
In addition, we have
$\ri(\Omega^2_{R_{\bbX}/K})=\ri(\Omega^2_{R_{\bbY}/K})\!=6$,
$\ri(\Omega^3_{R_{\bbX}/K})\!=4$, and
$\ri(\Omega^3_{R_{\bbY}/K})\!=5$. In this case,
the inequality of regularity indices in
Corollary~\ref{S1.Cor.03} is a strict inequality.
Moreover, the Hilbert functions of the exterior powers
$\Omega^2_{R_\bbX/K}$  and $\Omega^3_{R_\bbX/K}$ distinguish
a set $\bbX$ of six points on an irreducible conic from
a set $\bbY$ of six points on a reducible conic.
\end{exam} 

%% file: Sec2.tex
\section{K{\"a}hler Differential Algebras
for Subschemes of $\mathbb{P}^1$}

In this section we consider the easiest case, namely 
0-dimensional subschemes $\bbX$ of $\bbP^1$.
It is well known that
Hilbert functions do not change under base field
extensions  (for instance, see \cite[5.1.20]{KR2}).
Thus, in order to compute
the Hilbert function of the K{\"a}hler differential
algebra for the $0$-dimensional scheme $\bbX$
of~$\bbP^1$, we may assume that the field~$K$
is algebraically closed.
In this case the homogeneous vanishing ideal
$I_\bbX$ is a principal ideal generated by
a homogeneous polynomial $F \in S=K[X_0,X_1]$.
Moreover, after a suitable change of coordinates,
we may also assume that $F$ is of the form
$F=\prod_{i=1}^s(X_1-a_iX_0)^{m_i}$
where $s\ge 1$, $m_1,\dots,m_s\ge 1$ and
$a_1,\dots,a_s \in K$
such that $a_i\ne a_j$ for $i\ne j$.

In~\cite[Section~4]{Ro}, L.G. Roberts gave
a formula for the Hilbert function
of~$\Omega_{R_{\bbX}/K}$ when $m_1=m_2=\cdots=m_s=1$.
Now we extend his result to arbitrary exponents
$m_1,\dots, m_s\ge 1$ as follows.

\begin{prop} \label{S2.Prop.01}
Let $\bbX \subseteq \bbP^1$ be a $0$-dimensional scheme,
and let $I_{\bbX}=\langle F \rangle$, where
$F = \prod_{i=1}^s(X_1-a_iX_0)^{m_i}$ for some $s$,
$m_1,\dots,m_s\ge 1$, and $a_i \in K$ with $a_i\ne a_j$
for $i\ne j$, and let $\mu = \sum_{i=1}^sm_i$.
Then the Hilbert functions of the K{\"a}hler differential
modules of~$R_\bbX/K$ are given by
$$
\begin{array}{ll}
\HF_{\Omega^1_{R_{\bbX}/K}}&: \ 0 \ 2 \ 4 \ 6 \
    \cdots \ 2(\mu-2) \ 2(\mu-1) \ 2\mu-1
    \ 2\mu-2 \ \cdots \ 2\mu-s \ 2\mu-s \cdots \\
\HF_{\Omega^2_{R_{\mathbb{X}}/K}}&: \ 0 \ 0 \ 1 \ 2 \
    \cdots \ \mu-2 \ \mu-1 \ \mu-2 \ \mu- 3
    \ \cdots \ \mu- s \ \mu- s \cdots \\
\end{array}
$$
In particular, we have
$\ri(\Omega^1_{R_{\bbX}/K})=\ri(\Omega^2_{R_{\bbX}/K})
= \mu + s - 1$.
\end{prop}

\begin{proof}
Let $G = \prod_{i=1}^s(X_1-a_iX_0)^{m_i-1}$,
let $H_1 = \sum_{i=1}^sm_ia_i\prod_{j\ne i}(X_1-a_jX_0)$,
and let $H_2 = \sum_{i=1}^sm_i\prod_{j\ne i}(X_1-a_jX_0)$.
Note that $\deg(G)= \sum_{i=1}^s (m_i-1)$ and
$\deg(H_1) = \deg(H_2)=s-1$.
We verify that $\gcd(H_1,H_2)=1$.
Suppose for a contradiction that
$\gcd(H_1,H_2)=H$ with $\deg(H)\ge 1$.
Euler's relation $\mu F = G(-X_0H_1+ X_1H_2)$
implies $\mu\prod_{i=1}^s(X_1-a_iX_0) = -X_0H_1 + X_1H_2$.
So, $H$ is a divisor of~$\prod_{i=1}^s(X_1-a_iX_0)$.
There exists an index $i\in\{1,\dots,s\}$ such that
$(X_1-a_iX_0) \mid H$, but $(X_1-a_iX_0) \nmid H_2$,
a contradiction. Hence we obtain $\gcd(H_1,H_2)=1$.
Thus the sequence $\{H_1, H_2\}$ is an $S$-regular sequence.
Consequently, this sequence is also a regular
sequence for the principal ideal~$\langle G \rangle$
which is regarded as a graded $S$-module.
So, for $i\in\bbZ$, we have
$$
\begin{aligned}
\HF_{\Omega^2_{R_{\bbX}/K}}(i)
&= \HF_{S/\langle\frac{\partial F}{\partial X_0},
\frac{\partial F}{\partial X_1}\rangle }(i-2)\\
&= \HF_{S/\langle GH_1, GH_2 \rangle }(i-2) \\
&= \HF_{S/\langle G\rangle }(i-2)
+\HF_{\langle G \rangle /\langle GH_1,GH_2\rangle}(i-2)\\
&= \HF_{S/\langle G\rangle }(i-2)
+\HF_{\langle G\rangle }(i-2)
- 2\HF_{\langle G\rangle }(i-1-s)
+\HF_{\langle G\rangle }(i-2s) \\
&= \HF_{S}(i-2)-2\HF_{S}(i-1-\mu)
+\HF_{S}(i-s-\mu) \\
&= \binom{i-1}{1}-2\binom{i-\mu}{1}
+ \binom{i-s-\mu+1}{1}
\end{aligned}
$$
Thus the Hilbert function of $\Omega^2_{R_{\bbX}/K}$ is
$$
\HF_{\Omega^2_{R_{\mathbb{X}}/K}}: \ 0 \ 0 \ 1 \ 2
\ \cdots \ \mu-2 \ \mu-1 \ \mu-2 \ \mu- 3 \ \cdots
\ \mu- s \ \mu- s \cdots.
$$
Moreover, it is clear that $\HF_{\mathfrak{m}_{\bbX}}:
\ 0 \ 2 \ 3 \ 4 \ \cdots \ \mu-1 \ \mu \ \mu \cdots$.
By Proposition~\ref{S1.Prop.02}, we have the exact
sequence of graded $R_{\bbX}$-modules
$$
0\longrightarrow \Omega^2_{R_{\bbX}/K}
\longrightarrow \Omega^1_{R_{\bbX}/K}
\longrightarrow \mathfrak{m}_{\bbX}
\longrightarrow 0.
$$
Hence the Hilbert function of~$\Omega^1_{R_{\bbX}/K}$ satisfies
$\HF_{\Omega^1_{R_{\bbX}/K}}(i) =
\HF_{\mathfrak{m}_{\bbX}}(i)+\HF_{\Omega^2_{R_{\bbX}/K}}(i)$
for all $i\in\bbZ$.
More precisely, it is of the form
$$
\HF_{\Omega^1_{R_{\bbX}/K}}: 0 \ 2 \ 4 \ 6 \ \cdots
\ 2(\mu-2) \ 2(\mu-1) \ 2\mu-1
\ 2\mu-2 \ \cdots \ 2\mu-s \ 2\mu-s \cdots
$$
as claimed.
\end{proof}

Let us apply this proposition in an explicit example.

\begin{exam}
Let $\bbX\subseteq \bbP^1$ be the 0-dimensional scheme
with the homogeneous vanishing ideal
$I_\bbX = \langle X_1(X_1-X_0)^2(X_1-2X_0)^3\rangle$.
Clearly, we have $s=3$ and $\mu = 6$ and
$\HF_{\mathfrak{m}_\bbX}:\ 0\ 2\ 3\ 4\ 5\ 6\ 6\cdots$.
An application of Proposition~\ref{S2.Prop.01} yields
$$
\begin{array}{ll}
\HF_{\Omega^1_{R_{\bbX}/K}}&:
\ 0 \ 2\ 4\ 6\ 8\ 10\ 11\ 10\ 9\ 9\ \cdots \\
\HF_{\Omega^2_{R_{\mathbb{X}}/K}}&:
\ 0\ 0\ 1\ 2\ 3\ 4\ 5\ 4\ 3\ 3\ \cdots. \\
\end{array}
$$
In this case we have
$\ri(\Omega^1_{R_{\bbX}/K})=\ri(\Omega^2_{R_{\bbX}/K})= \mu + s - 1 =8$.
\end{exam} 

%% file: Sec3.tex
\section{Special Values of the Hilbert Function of~$\Omega^m_{R_{\bbX}/K}$}

In this section we describe the values of the Hilbert function
of the module of K\"{a}hler differential
$m$-forms for a $0$-dimensional scheme $\bbX$ at
some special degrees.
From now on, the coordinates $\{X_0,\dots,X_n\}$ of $\bbP^n$
are always chosen such that no point of $\bbX$ lies on
the hyperplane $\mathcal{Z}^+(X_0)$.
By the choice of the coordinates, $x_0$ is a non-zerodivisor
of $R_{\bbX}$. Moreover, $x_0$ is also a non-zerodivisor
for any nontrivial graded submodule of a graded free
$R_{\bbX}$-module, as the following lemma shows.

\begin{lem}\label{S3.Lem.01}
Let $d \ge 1$, let $\delta_1, \dots, \delta_d \in \bbZ$,
and let $V$ be a non-trivial graded submodule
of the graded free $R_{\bbX}$-module
$\bigoplus_{j=1}^d R_{\bbX}(-\delta_j)$.
Then $x_0$ is not a zerodivisor for~$V$, i.e.,
if $x_0\cdot v = 0$ for some $v \in V$ then $v=0$.
\end{lem}

\begin{proof}
Let $\{e_1,\dots,e_d\}$ be the canonical $R_{\bbX}$-basis
of $\bigoplus_{j=1}^d R_{\bbX}(-\delta_j)$,
and let $i \in \bbZ$. Then every homogeneous element
$v \in V_i$ has a representation $v = g_1e_1+\cdots+g_de_d$
for some homogeneous elements $g_1, \dots, g_d \in R_{\bbX}$,
where $\deg(g_j) = \deg(v)-\delta_j$ for $j = 1, \dots ,d$.
Suppose that $x_0\cdot v = 0$.
This implies that $x_0g_1e_1+\cdots+x_0g_de_d = 0$, and so
$x_0g_1 = \cdots = x_0g_d = 0$ in~$R_{\bbX}$.
Since $x_0$ is a non-zerodivisor for~$R_{\bbX}$,
we have $g_1=\cdots=g_d=0$, and hence $v =0$.
Thus the claim follows.
\end{proof}

The number $\alpha_{\bbX}=\min\{i\in \bbN | (I_{\bbX})_i\ne 0\}$
is called the \textbf{initial degree} of $I_{\bbX}$.
Using this notation, some basic properties
of $\HF_{\Omega^m_{R_{\bbX}/K}}$ can be summarized as follows.

\begin{prop}\label{S3.Prop.02}
Let $\bbX \subseteq \bbP^n_K$ be a $0$-dimensional scheme,
and let $1\le m\le n+1$.

\begin{enumerate}
\item For $i<m$, we have $\HF_{\Omega^m_{R_{\bbX}/K}}(i)=0$.
\item For $m\leq i<\alpha_{\bbX}+m-1$, we have
    $\HF_{\Omega^m_{R_{\bbX}/K}}(i)=\binom{n+1}{m}\cdot\binom{n+i-m}{n}$.
\item The Hilbert polynomial of~$\Omega^m_{R_{\bbX}/K}$ is constant.
\item  We have $\HF_{\Omega^m_{R_{\bbX}/K}}(r_{\bbX} + m)\geq
    \HF_{\Omega^m_{R_{\bbX}/K}}(r_{\bbX} + m + 1) \ge \cdots$, and if
    $\ri(\Omega^m_{R_{\bbX}/K}) \ge r_{\bbX} + m$ then
    $$
    \qquad
    \HF_{\Omega^m_{R_{\bbX}/K}}(r_{\bbX} + m) >
    \HF_{\Omega^m_{R_{\bbX}/K}}( r_{\bbX}+m+1)
    > \cdots >
    \HF_{\Omega^m_{R_{\bbX}/K}} (\ri(\Omega^m_{R_{\bbX}/K})).
    $$
\end{enumerate}
\end{prop}

\begin{proof}
(a)\ Obivously, every non-zero homogeneous element $\omega$
of~$\Omega^m_{R_{\bbX}/K}$ has degree $\deg(\omega)\ge m$,
and hence $\HF_{\Omega^m_{R_{\bbX}/K}}(i)=0$ for all $i<m$.

(b)\ Let $m\leq i<\alpha_{\bbX}+m-1$. Notice that
$I_{\bbX}\Omega^m_{S/K} \subseteq dI_{\bbX}\Omega^{m-1}_{S/K}$.
Also, we have $(dI_{\bbX}\Omega^{m-1}_{S/K})_i= \langle 0\rangle$
for all $i<\alpha_{\bbX}+m-1$, since a non-zero homogeneous
element of $dI_{\bbX}\Omega^{m-1}_{S/K}$ is always of the form
$\sum_k dF_k\wedge\omega_k$, where
$F_k \in (I_{\bbX})_{\ge \alpha_{\bbX}}$
and $\omega_k \in (\Omega^{m-1}_{S/K})_{\ge m-1}$.
By Proposition~\ref{S1.Prop.01}, for all $i<\alpha_{\bbX}+m-1$,
we obtain
$$
\HF_{\Omega^m_{R_{\bbX}/K}}(i) = \HF_{\Omega^m_{S/K}}(i) =
\binom{n+1}{m}\cdot\binom{n+i-m}{n}.
$$

(c)\ It follows from Proposition~\ref{S1.Prop.01} that
$$
\HF_{\Omega^m_{R_{\bbX}/K}}(i)\le
\HF_{\Omega^m_{S/K}/ I_{\bbX}\Omega^m_{S/K}}(i)
=\binom{n+1}{m}\HF_{\bbX}(i)
\le \binom{n+1}{m}\deg(\bbX)
$$
for all $i\in \bbZ$. Hence the Hilbert polynomial
of~$\Omega^m_{R_{\bbX}/K}$ is a constant polynomial.

(d)\ The graded $R_{\bbX}$-module $\Omega^m_{R_{\bbX}/K}$
has the following form:
$$
(\Omega^m_{R_{\bbX}/K})_{i+m}
=(R_{\bbX})_i dx_{0}\wedge\cdots\wedge dx_{m-1}
+ \cdots +
(R_{\bbX})_i dx_{n-m+1}\wedge\cdots\wedge dx_n.
$$
Observe that $(R_{\bbX})_i = x_0(R_{\bbX})_{i-1}$ if
$i > r_{\mathbb{X}}$. Thus
$(\Omega^m_{R_{\bbX}/K})_{i+m} =
x_0 (\Omega^m_{R_{\bbX}/K})_{i+m-1}$
for all $i > r_{\bbX}$. So, for all $i > r_{\bbX}$,
we have the inequality
$$
\HF_{\Omega^m_{R_{\bbX}/K}}(i+m-1)
\ge\HF_{\Omega^m_{R_{\bbX}/K}}(i+m).
$$

Now let $\mathcal{G}=\langle(
\frac{\partial F}{\partial x_0},\dots,
\frac{\partial F}{\partial x_n})
\mid F \in I_{\bbX} \rangle_{R_{\bbX}}$.
By \cite[Proposition~1.3]{DK} and \cite[X.83]{SS},
there is an exact sequence of graded $R_{\bbX}$-modules
$$
\textstyle{
0\longrightarrow \mathcal{G}\wedge_{R_{\bbX}}
\bigwedge^{m-1}_{R_{\bbX}} ({R_{\bbX}}^{n+1})
\longrightarrow \bigwedge^m_{R_{\bbX}} (R_{\bbX}^{n+1})
\longrightarrow \Omega^m_{R_{\bbX}/K}(m)
\longrightarrow 0.}
$$
Suppose $i \ge r_{\bbX}$ satisfies
$\HF_{\Omega^m_{R_{\bbX}/K}}(i+m)
=\HF_{\Omega^m_{R_{\bbX}/K}}(i+m+1)$.
Then it follows from the above exact sequence that
$$
\begin{aligned}
\HF_{\bigwedge^m_{R_{\bbX}} (R_{\bbX}^{n+1})}(i)
& -\HF_{\mathcal{G}\wedge_{R_{\bbX}}
\bigwedge^{m-1}_{R_{\bbX}} ({R_{\bbX}}^{n+1})}(i)\\
&= \HF_{\bigwedge^m_{R_{\bbX}} (R_{\bbX}^{n+1})}(i+1)
- \HF_{\mathcal{G}\wedge_{R_{\bbX}}
\bigwedge^{m-1}_{R_{\bbX}} (R_{\bbX}^{n+1})}(i+1).
\end{aligned}
$$
For every $j\ge r_\bbX$, $\HF_{\bbX}(j) = \deg(\bbX)$, and so
$\HF_{\bigwedge_{R_{\bbX}}^m (R_{\bbX}^{n+1})}(j)
=\HF_{\bigwedge_{R_{\bbX}}^m (R_{\bbX}^{n+1})}(j+1)$.
Consequently, we have
$$
\HF_{\mathcal{G}\wedge_{R_{\bbX}}
\bigwedge^{m-1}_{R_{\bbX}} (R_{\bbX}^{n+1})}(i) =
\HF_{\mathcal{G}\wedge_{R_{\bbX}}
\bigwedge^{m-1}_{R_{\bbX}} (R_{\bbX}^{n+1})}(i+1).
$$
In addition, Lemma~\ref{S3.Lem.01} shows that $x_0$ is
a non-zerodivisor for the graded $R_{\bbX}$-submodule
$\mathcal{G}\wedge_{R_{\bbX}}
\bigwedge^{m-1}_{R_{\bbX}}(R_{\bbX}^{n+1})$
of the graded-free $R_{\bbX}$-module
$\bigwedge^{m}_{R_{\bbX}}(R_{\bbX}^{n+1})$.
This implies
$$
(\mathcal{G}\wedge_{R_{\bbX}}
{\textstyle\bigwedge^{m-1}_{R_{\bbX}} }(R_{\bbX}^{n+1}))_{i+1}
=x_0(\mathcal{G}\wedge_{R_{\bbX}}
{\textstyle\bigwedge^{m-1}_{R_{\bbX}}} (R_{\bbX}^{n+1}))_{i}.
$$
In view of \cite[Proposition~1.1]{GM}, the ideal $I_{\bbX}$
can be generated by homogeneous polynomials of degrees
$\le r_{\bbX}+1$.
So, the graded $R_{\bbX}$-module
$\mathcal{G}\wedge_{R_{\bbX}}
\bigwedge^{m-1}_{R_{\bbX}}(R_{\bbX}^{n+1})$
is generated in degrees $\le r_{\bbX}$. Thus we obtain
$$
\begin{aligned}
(\mathcal{G}\wedge_{R_{\bbX}} &
 {\textstyle\bigwedge^{m-1}_{R_{\bbX}}} (R_{\bbX}^{n+1}))_{i+2}\\
 &=x_0(\mathcal{G}\wedge_{R_{\bbX}}
 {\textstyle\bigwedge^{m-1}_{R_{\bbX}}} (R_{\bbX}^{n+1}))_{i+1}+
 \cdots+ x_n(\mathcal{G}\wedge_{R_{\bbX}}
 {\textstyle\bigwedge^{m-1}_{R_{\bbX}}} (R_{\bbX}^{n+1}))_{i+1} \\
 &=x_0\big(x_0(\mathcal{G}\wedge_{R_{\bbX}}
 {\textstyle\bigwedge^{m-1}_{R_{\bbX}}} (R_{\bbX}^{n+1}))_i +
 \cdots+ x_n(\mathcal{G}\wedge_{R_{\bbX}}
 {\textstyle\bigwedge^{m-1}_{R_{\bbX}}} (R_{\bbX}^{n+1}))_i\big) \\
 &=x_0(\mathcal{G}\wedge_{R_{\bbX}}
 {\textstyle\bigwedge^{m-1}_{R_{\bbX}}} (R_{\bbX}^{n+1}))_{i+1}.
\end{aligned}
$$
Altogether, we have
$\HF_{\Omega^m_{R_{\bbX}/K}}(i+m+1)
=\HF_{\Omega^m_{R_{\bbX}/K}}(i+m+2)$,
and the claim follows by induction.
\end{proof}

The following example shows that $\HF_{\Omega^m_{R_{\bbX}/K}}(i)$
may or may not be monotonic in the range
$\alpha_{\mathbb{X}}+m \le i\leq r_{\bbX}+m$.

\begin{exam}
Let $K=\bbQ$, and let $\bbX\subseteq\mathbb{P}^2$
be the set of nine points
$\bbX=\{(1:1:0),$ $(1:1:1), (1:1:2), (1:1:3), (1:1:4),
(1:1:5), (1:0:1), (1:2:1), (1:2:2)\}$.
Notice that $\bbX$ contains six points on a line
and three non-collinear points off that line.
It is clear that
$\HF_{\bbX}:\ 1 \ 3 \ 6 \ 7 \ 8 \ 9 \ 9 \cdots$,
$\alpha_\bbX = 3$, and $r_\bbX = 5$.
The Hilbert functions of the K{\"a}hler differential modules
of~$R_{\bbX}/K$ are given by
$$
\begin{aligned}
\HF_{\Omega^1_{R_{\bbX}/K}}&:  \ 0 \ 3 \ 9 \ 15 \ 14 \ 13 \
14 \ 13 \ 12 \ 11 \ 10 \ 9 \ 9 \cdots \\
\HF_{\Omega^2_{R_{\bbX}/K}}&:  \ 0 \ 0 \ 3 \ 9 \ 9 \ 4 \
5 \ 4 \ 3 \ 2 \ 1 \ 0 \ 0 \cdots \ \\
\HF_{\Omega^3_{R_{\bbX}/K}}&:  \ 0 \ 0 \ 0 \ 1 \ 3 \ 0 \ 0 \cdots.
\end{aligned}
$$
We see that $\HF_{\Omega^1_{R_{\bbX}/K}}(\alpha_{\mathbb{X}}+1) = 14 >
13 = \HF_{\Omega^1_{R_{\bbX}/K}}(\alpha_{\mathbb{X}}+2)$
and $\HF_{\Omega^1_{R_{\bbX}/K}}(\alpha_{\mathbb{X}}+2) =$
$ 13 < 14 = \HF_{\Omega^1_{R_{\bbX}/K}}(r_{\mathbb{X}}+1)$.
So, $\HF_{\Omega^1_{R_{\bbX}/K}}(i)$ is not monotonic in the
range  $\alpha_{\mathbb{X}}+1 \le i \le r_{\bbX}+1$.
Similarly, $\HF_{\Omega^2_{R_{\bbX}/K}}(i)$ is not monotonic
in the range $\alpha_{\mathbb{X}}+2 \le i \le r_{\bbX}+2$.

Next we consider the set
$\bbY = \bbX\cup\{(1:0:2)\}$.
We have $\HF_{\bbY}: \ 1 \ 3 \ 6 \ 8 \ 9 \ 10 \ 10 \cdots$,
$\alpha_{\bbY} = 3$, and $r_{\bbY} = 5$.
The Hilbert functions of the K{\"a}hler differential modules
of~$R_{\bbY}/K$ are
$$
\begin{aligned}
\HF_{\Omega^1_{R_{\bbY}/K}}&:
\ 0 \ 3 \ 9 \ 16 \ 18 \ 16 \
15 \ 14 \ 13 \ 12 \ 11 \ 10 \ 10 \cdots \\
\HF_{\Omega^2_{R_{\bbY}/K}}&:
\ 0 \ 0 \ 3 \ 9 \ 12 \ 8 \
5 \ 4 \ 3 \ 2 \ 1 \ 0 \ 0 \cdots \ \\
\HF_{\Omega^3_{R_{\bbY}/K}}&:
\ 0 \ 0 \ 0 \ 1 \ 3 \ 2 \ 0 \cdots.
\end{aligned}
$$
Hence $\HF_{\Omega^1_{R_{\bbY}/K}}(i)$ is monotonic
in the range $\alpha_{\bbY}+1 \le i \le r_{\bbY}+1$, and
$\HF_{\Omega^2_{R_{\bbY}/K}}(i)$ is also monotonic
in the range $\alpha_{\bbY}+2 \le i \le r_{\bbY}+2$.
\end{exam} 

%% file: Sec4.tex
\section{Bounds for the Regularity Index of~$\Omega^m_{R_{\bbX}/K}$}

In this section we give an upper bound for the regularity
index of the module of K\"{a}hler differential $m$-forms
$\Omega^m_{R_{\bbX}/K}$
for a 0-dimensional scheme $\bbX$ in $\bbP^n$.
To do this, we need the following lemmas.

\begin{lem}\label{S4.Lem.01}
Let $d\ge 1$, let $\delta_1, \dots, \delta_d \in \bbZ$ such that
$\delta_1 \le \cdots \le \delta_d$, let
$W = \bigoplus_{j=1}^d R_{\bbX}(-\delta_j)$
be the graded free $R_{\bbX}$-module, and let $V$ be
a non-trivial graded submodule of~$W$.
Then, for  $1 \le m \le d$, we have
$$
\ri(V\wedge_{R_{\bbX}} {\textstyle \bigwedge_{R_{\bbX}}^{m}}(W))
\le \ri(V) + \delta_{d-m+1} + \dots + \delta_d.
$$
\end{lem}

\begin{proof}
First we note that the Hilbert polynomial of~$W$ is
$\HP_{W}(z) = d\cdot\deg(\bbX)$ and
that~$\ri(W)=r_{\bbX}+\delta_d$.
This shows that the Hilbert polynomial of~$V$ is
a constant polynomial
$\HP_{V}(z) = u \le d\cdot\deg(\bbX)$. Let $r = \ri(V)$,
and let $v_1,\dots,v_u$ be a $K$-basis of~$V_r$.
By Lemma~\ref{S3.Lem.01}, the elements
$\{x_0^iv_1,\dots,x_0^iv_u\}$
form a~$K$-basis of the $K$-vector space $V_{r+i}$ for all
$i\in \bbN$. We let~$\{e_1,\dots,e_d\}$ be the canonical
$R_{\bbX}$-basis of~$W$, we let~$t=\binom{d}{m}$, and we let
$\{\varepsilon_1,\dots,\varepsilon_t\}$ be a basis of the
graded free $R_{\bbX}$-module $\bigwedge_{R_{\bbX}}^{m}(W)$
w.r.t.~$\{e_1,\dots,e_d\}$.
We set~$\delta =  \delta_{d-m+1} +\dots+ \delta_d$, and let
$$
N = \langle x_0^{\delta-\deg(\varepsilon_k)}v_j
\wedge \varepsilon_k \in  V \wedge_{R_{\bbX}}
{\textstyle \bigwedge_{R_{\bbX}}^{m}} (W)
\mid 1 \le j \le u, 1 \le k \le t \rangle_K.
$$
Let $\varrho = \dim_KN$, and let $w_1,\dots, w_\varrho$
be a $K$-basis of~$N$. It is not difficult to verify that
$N = \langle w_1, \dots, w_{\varrho}\rangle_K
= (V \wedge_{R_{\bbX}} {\textstyle
\bigwedge_{R_{\bbX}}^{m}} (W))_{\delta+r}$.
Moreover, for any $i \ge 0$, the set
$\{x_0^iw_1, \dots, x_0^iw_\varrho\}$ is $K$-linearly
independent. Indeed, assume that there are elements
$a_1,\dots,a_{\varrho}\in K$ such that
$\sum_{j=1}^{\varrho} x_0^ia_jw_j = 0$.
Since $x_0$ is a non-zerodivisor
for~$V\wedge_{R_{\bbX}}\bigwedge_{R_{\bbX}}^{m}(W)$
by Lemma~\ref{S3.Lem.01}, we get
$\sum_{j=1}^{\varrho} a_jw_j = 0$, and hence
$a_1=\cdots=a_{\varrho}=0$.

Now it suffices to prove that the set
$\{x_0^iw_1, \dots, x_0^iw_\varrho\}$ generates
the $K$-vector space
$(V\wedge_{R_{\bbX}} \bigwedge_{R_{\bbX}}^{m}
(W))_{\delta + r +i}$ for all $i \ge 0$.
Let $w\in(V\wedge_{R_{\bbX}}
\bigwedge_{R_{\bbX}}^{m}(W))_{\delta+r+i}$
be a non-zero homogeneous element.
Then $w = \sum_{j,k} \widetilde{v}_j \wedge h_k\varepsilon_k
= \sum_{j,k} h_k\widetilde{v}_j \wedge \varepsilon_k$
for some homogeneous elements $\widetilde{v}_j\in V$
and $h_k\in R_{\bbX}$ such that
$\deg(\widetilde{v}_j)+\deg(h_k)=\delta+r+i-\deg(\varepsilon_k)$
for all $j, k$. Note that
$\deg(h_k\widetilde{v}_j)=\delta+r+i-\deg(\varepsilon_k)\ge r+i$.
Also, we have
$$
h_k\widetilde{v}_j \in V_{\delta + r +i - \deg(\varepsilon_k)}
= \langle x_0^{\delta + i - \deg(\varepsilon_k)}v_1,\dots,
x_0^{\delta + i - \deg(\varepsilon_k)}v_u \rangle_K.
$$
So, there are $b_{jk1}, \dots, b_{jku} \in K$ such that
$h_k\widetilde{v}_j = \sum_{l=1}^u b_{jkl}x_0^{\delta +
i - \deg(\varepsilon_k)}v_l$.
This implies
$$
\begin{aligned}
w  &= {\textstyle\sum\limits_{j,k}} h_k\widetilde{v}_j
      \wedge \varepsilon_k
    = {\textstyle\sum\limits_{j,k} \sum\limits_{l=1}^u}
      b_{jkl}x_0^{\delta + i - \deg(\varepsilon_k)} v_l
      \wedge \varepsilon_k \\
   &= {\textstyle\sum\limits_{j,k} \sum\limits_{l=1}^u}
      b_{jkl}x_0^i (x_0^{\delta - \deg(\varepsilon_k)} v_l
      \wedge \varepsilon_k)
    = {\textstyle\sum\limits_{j,k}
      \sum\limits_{l=1}^u \sum\limits_{q=1}^{\varrho}}
      b_{jkl} c_{jklq} x_0^iw_q
\end{aligned}
$$
for some $c_{jklq} \in K$.  Thus we get
$w \in \langle x_0^iw_1, \dots, x_0^iw_\varrho \rangle_K$,
and consequently
$\HF_{V\wedge_{R_{\bbX}} \bigwedge_{R_{\bbX}}^{m} (W)}(i)
= \varrho$ for all $i \ge \delta + r$.
Therefore $\ri(V\wedge_{R_{\bbX}} \bigwedge_{R_{\bbX}}^{m} (W))
\le \ri(V) + \delta$, as we wanted to show.
\end{proof}

\begin{lem}\label{S4.Lem.02}
Let $V$ be a graded $R_{\bbX}$-module generated by the set
of homogeneous elements $\{v_1, \dots, v_d\}$ for some
$d\ge 1$. Let $\delta_j = \deg(v_j)$ for $j=1, \dots, d$,
and let~$m \ge 1$.
Assume that $\delta_1 \le \cdots \le \delta_d$, and set
$\delta = \delta_{d-m+1} +\cdots+ \delta_d$ if $m \le d$.
Then the regularity index of~$\bigwedge_{R_{\bbX}}^m(V)$ satisfies
$\ri(\bigwedge_{R_{\bbX}}^m(V)) = -\infty$ if $m> d$ and
$$
\ri({\textstyle\bigwedge_{R_{\bbX}}^m}(V)) \le
\max\big\{\, r_{\bbX} + \delta+\delta_d-\delta_{d-m+1},
\, \ri(V) + \delta - \delta_{d-m+1} \,\big\}
$$
if $1 \le m \le d$. In particular, if $1\le m\le d$ and
$\delta_1 = \cdots = \delta_d = t$ then we have
$\ri({\textstyle\bigwedge_{R_{\bbX}}^m}(V)) \le
\max\{\, r_{\bbX} + mt, \ri(V) + (m-1)t \,\}$.
\end{lem}

\begin{proof}
If $m >d$, then  $\bigwedge_{R_{\bbX}}^{m}(V)=\langle0\rangle$,
and hence $\ri(\bigwedge_{R_{\bbX}}^m(V)) = -\infty$.
Now we assume that $1\le m \le d$.
Obviously, the $R_{\bbX}$-linear map
$\alpha: W=\bigoplus_{j=1}^d R_{\bbX}(-\delta_j)\rightarrow V$
given by $e_j\mapsto v_j$
is a homogeneous $R_{\bbX}$-epimorphism of degree zero.
Set $\mathcal{G} = \Ker(\alpha)$.
According to \cite[X.83]{SS}, there is
an exact sequence of graded $R_{\bbX}$-modules
$$
0\longrightarrow \mathcal{G}\wedge_{R_{\bbX}}
{\textstyle\bigwedge^{m-1}_{R_{\bbX}}}(W)
\longrightarrow {\textstyle\bigwedge^{m}_{R_{\bbX}}}(W)
\stackrel{\bigwedge^m(\alpha)}{\longrightarrow}
{\textstyle\bigwedge^{m}_{R_{\bbX}}}(V)\longrightarrow 0.
$$
Thus an application of Lemma~\ref{S4.Lem.01} yields that
$$
\begin{aligned}
\ri({\textstyle\bigwedge^m_{R_\bbX}}(V))
&\leq \max\big\{\, \ri({\textstyle\bigwedge^{m}_{R_{\bbX}}}(W)),
\,\ri(\mathcal{G}\wedge_{R_{\bbX}}
{\textstyle\bigwedge^{m-1}_{R_{\bbX}}}(W)) \,\big\}\\
&\leq \max\big\{\, r_{\bbX} + \delta,\, \ri(\mathcal{G})+
\delta - \delta_{d-m+1} \,\big\}\\
&\leq \max\big\{\, r_{\bbX} + \delta+\delta_d-\delta_{d-m+1},
\,  \ri(V) +\delta - \delta_{d-m+1} \,\big\}.
\end{aligned}
$$
Here the last inequality follows from the fact that
$\ri(\mathcal{G}) \le \max\{r_{\bbX}+ \delta_d, \ri(V)\}$.
\end{proof}

Now we are able to give an upper bound for the regularity
index of the module of K\"{a}hler differential $m$-forms
$\Omega^m_{R_{\bbX}/K}$.

\begin{prop}\label{S4.Prop.03}
Let $\bbX \subseteq \bbP^n$ be a $0$-dimensional scheme,
and let $1\le m \le n+1$.
The regularity index of the module of K\"{a}hler
differential $m$-forms $\Omega^m_{R_{\bbX}/K}$ satisfies
$$
\ri(\Omega^m_{R_{\bbX}/K})\leq
\max\{r_{\bbX}+m, \ri(\Omega^1_{R_{\bbX}/K})+m-1\}.
$$
\end{prop}

\begin{proof}
We set $\mathcal{G} = \langle (\frac{\partial F}{\partial x_0},
\dots,\frac{\partial F}{\partial x_n}) \in R_\bbX^{n+1}
\mid F \in I_\bbX\rangle$.
By~\cite[Proposition~1.3]{DK}, we have the short exact sequence
of graded $R_{\bbX}$-modules
$$
0\longrightarrow \mathcal{G}(-1)
\longrightarrow R_\bbX^{n+1}(-1)
\longrightarrow \Omega^1_{R_{\bbX}/K}
\longrightarrow 0.
$$
Applying Lemma~\ref{S4.Lem.02} to the graded $R_{\bbX}$-module
$\Omega^1_{R_{\bbX}/K}$ which is generated by the set
$\{dx_0,\dots,dx_n\}$, we get
$\ri(\Omega^m_{R_{\bbX}/K}) \le
\max\{r_{\bbX}+m, \ri(\Omega^1_{R_{\bbX}/K})+m-1\},$
as we wished.
\end{proof}

\begin{rem}\label{S4.Rem.04}
We have $\ri(\Omega^{n+1}_{R_{\bbX}/K}) \le
\max\{r_{\bbX}+n, \ri(\Omega^1_{R_{\bbX}/K})+n-1\}.$
Indeed, the exact sequence $(\mathcal{K})$ of graded 
$R_{\bbX}$-modules yields
$$
\ri(\Omega^{n+1}_{R_{\bbX}/K})\le
\max\{\ri(\Omega^i_{R_{\bbX}/K}) | i=0,\dots,n\}\le
\max\{r_{\bbX}+n, \ri(\Omega^1_{R_{\bbX}/K})+n-1\}.
$$
Moreover, if we set $\varrho_m =
\max\{r_{\bbX}+m, \ri(\Omega^1_{R_{\bbX}/K})+m-1\}\}$
for $m\ge 1$, then we get the upper bound for
the regularity index of $\Omega^m_{R_{\bbX}/K}$ as
$\ri(\Omega^m_{R_{\bbX}/K})\le \min\{\varrho_n, \varrho_m\}$.
\end{rem}

%% file: Sec5.tex
\section{Bounds for $\ri(\Omega^m_{R_{\bbW}/K})$
for a Fat Point Scheme $\bbW$}

Let $s \ge 1$, and let $\bbX=\{P_1,\dots,P_s\}$ be a set of
$s$ distinct $K$-rational points in~$\bbP^n$.
For $i=1,\dots,s$, we let $\wp_i$ be the associated prime
ideal of~$P_i$ in~$S$.

\begin{defn}
Given a sequence of positive integers $m_1,\dots,m_s$,
the intersection $I_\bbW:=\wp_1^{m_1}\cap\cdots\cap\wp_s^{m_s}$
is a saturated homogeneous ideal in~$S$ and is therefore
the vanishing ideal of a $0$-dimensional subscheme~$\bbW$
of~$\bbP^n$.

\begin{enumerate}
\item The scheme $\bbW$, denoted by $\bbW=m_1P_1+\cdots+m_sP_s$,
is called a {\bf fat point scheme} in~$\bbP^n$.
The homogeneous vanishing ideal of~$\bbW$ is $I_\bbW$.
The number $m_j$ is called the {\bf multiplicity} of
the point $P_j$ for $j=1, \dots, s$.

\item If $m_1=\cdots=m_s=\nu$, we denote $\bbW$ also
by $\nu\bbX$ and call it an
{\bf equimultiple fat point scheme}.

\item For $i\ge 1$, the fat point scheme
$\bbW^{(i)} = (m_1+i)P_1+\cdots+(m_s+i)P_s$
is called the {\bf $i$-th fattening} of~$\bbW$.
We simple say the fattening of~$\bbW$ instead of
the first fattening of~$\bbW$.
\end{enumerate}
\end{defn}

The regularity index of the module of K\"{a}hler differential
$m$-forms for fat point schemes can be bounded as follows.

\begin{prop}\label{S5.Prop.02}
Let $\bbW = m_1P_1+\cdots+m_sP_s$ be a fat point scheme
in~$\bbP^n$, and let $\bbV=\bbW^{(1)}$ be the fattening of~$\bbW$.

\begin{enumerate}
\item For $1\le m\le n+1$, we have
\[
\hspace*{1cm} \ri(\Omega^m_{R_{\bbW}/K})\le
\min\big\{\max \{r_{\bbW}+m, r_{\bbV}+m-1\},
\max\{r_{\bbW} + n,r_{\bbV}+n-1\}\big\}.
\]

\item If $m_1\le\cdots\le m_s$ and if
$\Supp(\bbW)=\{P_1,\dots,P_s\}$ is in general position,
then we have
$$
\begin{aligned}
\ri(\Omega^m_{R_{\bbW}/K})\le \min\big\{\,
&\max\{m_s+m_{s-1}+m,
\lfloor \tfrac{\sum_{j=1}^sm_j+s+n-2}{n}\rfloor+m-1\},\\
&\max\{m_s+m_{s-1}+n,
\lfloor\tfrac{\sum_{j=1}^sm_j+s+n-2}{n}\rfloor+n-1\}
\,\big\}
\end{aligned}
$$
for $1\le m\le n+1$.
\end{enumerate}
\end{prop}

\begin{proof}
Claim (a) follows from Remark~\ref{S4.Rem.04}
and \cite[Corollary~1.9(iii)]{KLL}.
Moreover, if $\Supp(\bbW)$ is in general position,
then \cite[Theorem~6]{CTV} implies that
$$
\begin{aligned}
\max\{r_{\bbW} + m, r_{\bbV} + m - 1\}
&\le \max\big\{\,  m_s+m_{s-1}+m-1,
\lfloor \tfrac{\sum_{j=1}^sm_j+n-2}{n}\rfloor+m,         \\
&\qquad\qquad m_s+m_{s-1}+m,
\lfloor \tfrac{\sum_{j=1}^sm_j+s+n-2}{n} \rfloor +m-1 \,\big\}\\
&\le \max\big\{\, m_s+m_{s-1}+m,
\lfloor \tfrac{\sum_{j=1}^sm_j+s+n-2}{n} \rfloor +m-1 \,\big\}.
\end{aligned}
$$
Thus claim (b) follows from (a).
\end{proof}

The following example shows that the upper bounds
for the regularity index of~$\Omega^{m}_{R_{\bbW}/K}$
given in Proposition~\ref{S5.Prop.02} are sharp.

\begin{exam}\label{S5.Exam.03}
Let $K=\bbQ$, and let $\bbW$ be the fat point scheme
$$
\bbW = P_1+2P_2+P_3+P_4+2P_5+2P_6+2P_7+P_8
\subseteq \bbP^3
$$
where $P_1=(1:9:0:0)$, $P_2=(1:6:0:1)$, $P_3=(1:2:3:3)$,
$P_4=(1:9:3:5)$, $P_5=(1:3:0:4)$, $P_6=(1:0:1:3)$,
$P_7=(1:0:2:0)$, and $P_8=(1:3:0:10)$.
Let $\bbV$ be the fat point scheme
$\bbV= 2P_1+3P_2+2P_3+2P_4+3P_5+3P_6+3P_7+2P_8$.
We have $r_{\bbW}=3$ and $r_{\bbV}=5$, and so
$\max\{r_{\bbW}+m, r_{\bbV}+m-1\}=m+4$ for $m=1,\dots,4$.
In this case the regularity index of~$\Omega^m_{R_{\bbW}/K}$
is $m+4$ for $m=1,\dots,3$ and $\ri(\Omega^4_{R_{\bbW}/K})=7$.
Thus the bound for the regularity index in
Proposition~\ref{S5.Prop.02}(a) is sharp.

Next let $\bbY$ be the scheme $\bbY=P_4+P_5+P_6+P_7+P_8$
in~$\bbP^3$. Then $\bbY$ is in general position.
For $m=1,2,3$, the regularity index of~$\Omega^m_{R_{2\bbY}/K}$
is $4+m$. Thus, for $m=1,2,3$, we have
$$
\ri(\Omega^m_{R_{2\bbY}/K}) = 4+m
=\max\big\{\, 2+2+m,
\lfloor\tfrac{\sum_{i=1}^5 2+5+3-2}{3} \rfloor+m-1
\,\big\}.
$$
In addition, for $m=4$, we have
$$
\ri(\Omega^4_{R_{2\bbY}/K}) = 7
= \max\big\{\,
2+2+3,\lfloor\tfrac{\sum_{i=1}^5 2+5+3-2}{3}\rfloor+3-1
\,\big\},
$$
and hence the bound in Proposition~\ref{S5.Prop.02}(b)
is also sharp.
\end{exam} 

%% file: Sec6.tex
\section{Bounds for the Hilbert Polynomial
of~$\Omega^m_{R_{\bbW}/K}$ for a Fat Point Scheme $\bbW$}

First we determine the Hilbert polynomial of the module
of K\"{a}hler differential $m$-forms for a set of $s$
distinct $K$-rational points $\bbX=\{P_1,\dots,P_s\}$
in~$\bbP^n$. Notice that all points of~$\bbX$ are assumed
to lie outside the hyperplane $\mathcal{Z}^+(X_0)$,
so we may write $P_j = (1: p_{j1}: \dots : p_{jn})$
with $p_{j1}, \dots, p_{jn}\in K$ for $j=1,\dots, s$.
Furthermore, for every element $f\in R_{\bbX}$
and $j\in\{1,\dots,s\}$, we also write $f(P_j)=F(P_j)$,
where $F$ is any representative of~$f$ in~$S$.

Recall that an element $f_j \in (R_{\bbX})_{r_{\bbX}}$
is the {\bf normal separator} of~$\bbX \setminus\{P_j\}$ in~$\bbX$
if $f_j(P_j)= 1$ and $f_j(P_k)= 0$ for $k\ne j$.
The set
$\{\, x_0^{i-r_{\bbX}}f_1, \dots, x_0^{i- r_{\bbX}}f_s \,\}$
is a $K$-basis of~$(R_\bbX)_i$ for all $i \ge r_{\bbX}$.
For more details about separators of a 0-dimensional scheme
in~$\bbP^n$ see \cite{GKR,GMT,Kre1,Kre2}.

The following proposition gives a description of
the Hilbert polynomial of the module of K\"{a}hler differential
$m$-forms $\Omega^m_{R_{\bbX}/K}$ for every $1 \le m \le n+1$.

\begin{prop}\label{S6.Prop.01}
Let $\bbX =\{P_1,\dots,P_s\}\subseteq \bbP^n$ be a set of $s$
distinct $K$-rational points, and let $1 \le m \le n+1$. We have
$$
\HP_{\Omega^m_{R_{\bbX}/K}}(z) =
\begin{cases}
\deg(\bbX) & \mbox{if $m = 1$}, \\
0          & \mbox{if $m \ge 2$}.
\end{cases}
$$
In particular, the regularity index of~$\Omega^m_{R_{\bbX}/K}$
satisfies $\ri(\Omega^m_{R_{\bbX}/K}) \le 2r_{\bbX} + m$.
\end{prop}

\begin{proof}
For $m=1$, we have $\HP_{\Omega^1_{R_{\bbX}/K}}(z)=\deg(\bbX)$ and
$\ri(\Omega^1_{R_{\bbX}/K})\le 2r_{\bbX}+1$
(see \cite[Proposition~3.5]{DK}). Assume that $m \ge  2$.
We see that $\Omega^m_{R_{\bbX}/K}$ is a graded $R_{\bbX}$-module
generated by the set of $\binom{n+1}{m}$ elements
$$
\big\{\, dx_{i_1}\wedge \cdots \wedge  dx_{i_m}
\mid 0\le i_1< \cdots< i_m\le n \,\big\}.
$$
For $j \in \{1,\dots,s\}$, let $f_j$ be the normal separator
of~$\bbX\setminus\{P_j\}$ in~$\bbX$.
Since the set $\{x_0^{i-r_{\bbX}} f_1, \dots, x_0^{i-r_{\bbX}} f_s\}$
is a $K$-basis of the $K$-vector space $(R_{\bbX})_i$ for
$i \ge r_{\bbX}$, the set
$$
\big\{\,
x_0^{k-r_{\bbX}-m} f_j dx_{i_1} \wedge\cdots\wedge dx_{i_m}
\mid 0\le i_1< \cdots< i_m \le n, 1\le j \le s
\,\big\}
$$
is a system of generators of the $K$-vector space
$(\Omega^m_{R_{\bbX}/K})_k$ for all $k \ge r_{\bbX} + m$.
Note that $f_j^2 = f_j(P_j)x_0^{r_{\bbX}}f_i = x_0^{r_{\bbX}}f_i$
and $x_if_j = p_{ji}x_0f_j$
(see, e.g., \cite[Proposition~1.13]{GKR}).
Therefore we get
$$
\begin{aligned}
  x_0^{r_{\bbX}}f_j dx_{i_1}\wedge\cdots\wedge dx_{i_m}
   & =   f_j^2 dx_{i_1}\wedge\cdots\wedge dx_{i_m} \\
   & = (d(f_j^2x_{i_1})- x_{i_1}df_j^2)
      \wedge dx_{i_2}\wedge\cdots\wedge dx_{i_m} \\
   & =   (d(p_{ji_1}x_0f_j^2) - x_{i_1} df_j^2)
        \wedge dx_{i_2}\wedge\cdots\wedge dx_{i_m}\\
   & =   ((p_{ji_1}x_0-x_{i_1}) df_j^2 + p_{ji_1}f_j^2 dx_0)
      \wedge dx_{i_2} \wedge\cdots\wedge dx_{i_m}\\
   & =   (2(p_{ji_1}x_0-x_{i_1})f_j df_j + p_{ji_1}f_j^2 dx_0)
      \wedge dx_{i_2} \wedge\cdots\wedge dx_{i_m}\\
   & =   p_{ji_1}f_j^2 dx_0 \wedge dx_{i_2}
      \wedge \cdots\wedge dx_{i_m} \\
   & = p_{ji_1} x_0^{r_{\bbX}}f_j dx_0
      \wedge dx_{i_2}\wedge\cdots \wedge dx_{i_m}.
\end{aligned}
$$
Since $m\ge 2$, we may use the same method as above to get
the equality
$$
x_0^{r_{\bbX}}f_j dx_{i_1}\wedge\cdots\wedge dx_{i_m}
=  p_{ji_1}p_{ji_2} x_0^{r_{\bbX}}
   f_j dx_0\wedge dx_0 \wedge dx_{i_3}
   \wedge\cdots\wedge dx_{i_m}.
$$
This implies
$x_0^{r_{\bbX}}f_j dx_{i_1}\wedge\cdots\wedge dx_{i_m} =0$
for $j =1, \dots, s$ and
$\{i_1, \dots, i_m\} \subseteq \{0, \dots, n\}$.
Thus we obtain
$(\Omega^m_{R_{\bbX}/K})_k = \langle 0 \rangle$
for all $k \ge 2r_{\bbX} + m$, and the claim follows.
\end{proof}

By combining Corollary~\ref{S1.Cor.03} and
Proposition~\ref{S6.Prop.01}, we obtain upper bounds
for the regularity indices of modules of the K{\"a}hler
differential $m$-forms $\Omega^{m}_{R_{\bbX}/K}$ as follows.

\begin{cor}\label{S6.Cor.02}
In the setting of Proposition~\ref{S6.Prop.01}, we have
$$
\ri(\Omega^{m}_{R_{\bbX}/K})\le \min\{\, 2r_{\bbX}+m, 2r_{\bbX}+n \,\}.
$$
\end{cor}

The preceding bounds for the regularity indices
of~$\Omega^{m}_{R_{\bbX}/K}$
are sharp, as our next example shows.

\begin{exam}\label{S6.Exam.02}
Let $K=\bbQ$, and let $\bbX \subseteq \bbP^3$ be the set of four
$K$-rational points $\bbX=\{P_1,P_2,P_3,P_4\}$, where
$P_1=(1:9:0:0)$, $P_2=(1:6:0:1)$, $P_3=(1:2:3:3)$, and $P_4=(1:9:3:5)$.
It is clear that $\HF_{\bbX}: 1 \ 4 \ 4 \cdots$
and $r_{\bbX}=1$. Moreover, we have
$$
\begin{aligned}
\HF_{\Omega^1_{R_{\bbX}/K}} &: 0 \ 4 \ 10\ 4 \ 4\cdots,\
&\HF_{\Omega^2_{R_{\bbX}/K}} &: 0 \ 0 \ 6\ 4\ 0\ 0\cdots,\\
\HF_{\Omega^3_{R_{\bbX}/K}} &: 0\ 0\ 0\ 4\ 1\ 0\ 0\cdots,\
&\HF_{\Omega^4_{R_{\bbX}/K}} &: 0\ 0\ 0\ 0\ 1\ 0\ 0\cdots.
\end{aligned}
$$
It follows that
$\ri(\Omega^1_{R_{\bbX}/K}) = 2r_{\bbX}+m = 2r_{\bbX}+1 = 3$,
$\ri(\Omega^2_{R_{\bbX}/K}) = 2r_{\bbX}+m = 2r_{\bbX}+2 = 4$,
and $\ri(\Omega^3_{R_{\bbX}/K}) =  \ri(\Omega^4_{R_{\bbX}/K})
= \min\{2r_{\bbX}+m, 2r_{\bbX}+n\} = 2r_{\bbX}+n = 5$.
Hence we obtain the equality
$\ri(\Omega^m_{R_{\bbX}/K}) = \min\{2r_{\bbX}+m, 2r_{\bbX}+n\}$
for $m=1,\dots,4$.
Consequently, the bounds in Corollary~\ref{S6.Cor.02} are sharp.
\end{exam}

Now we give bounds for the Hilbert polynomial of the module
of K\"{a}hler differential $m$-forms for a non-reduced fat
point scheme.

\begin{prop}\label{S6.Prop.03}
Let $\bbW = m_1P_1 + \cdots + m_sP_s$ be a fat point scheme
in~$\bbP^n$ such that $m_i \geq 2$ for some $i \in \{1,\dots,s\}$,
and let $1 \le m \le n+1$.
The Hilbert polynomial of~$\Omega^m_{R_{\bbW}/K}$
is a constant polynomial which is bounded by
$$
\sum_{i=1}^s\binom{n+1}{m}\binom{m_i+n-2}{n}
\le \HP_{\Omega^m_{R_{\bbW}/K}}(z)
\le \sum_{i=1}^s\binom{n+1}{m}\binom{m_i+n-1}{n}.
$$
\end{prop}

\begin{proof}
Let $\bbY$ be the subscheme
$\bbY= (m_1-1)P_1+\cdots+(m_s-1)P_s$ of~$\bbW$.
Since we have $dI_{\bbW} \subseteq I_{\bbY}\Omega^1_{S/K}$,
this implies
$dI_{\bbW}\Omega^{m-1}_{S/K}\subseteq I_{\bbY}\Omega^m_{S/K}$.
Obviously, we have the inclusion $I_{\bbW}\subseteq I_{\bbY}$,
and so $I_{\bbW}\Omega^m_{S/K}\subseteq I_{\bbY}\Omega^m_{S/K}$.
From this we deduce
$$
I_{\bbW}\Omega^m_{S/K}+ dI_{\bbW} \Omega^{m-1}_{S/K}
\subseteq  I_{\bbY}\Omega^m_{S/K}.
$$
By Proposition~\ref{S1.Prop.01},
the Hilbert function of~$\Omega^m_{R_{\bbW/K}}$ satisfies
$$
\HF_{\Omega^m_{R_{\bbW}/K}}(i)
= \HF_{\Omega^m_{S/K}/(I_{\bbW}\Omega^m_{S/K}
+ dI_{\bbW}\Omega^{m-1}_{S/K})} (i)
\ge \HF_{\Omega^m_{S/K}/I_{\bbY}\Omega^m_{S/K}}(i)
$$
for all $i \in \bbZ$.
Also, we see that
$\HP_{\Omega^m_{S/K}/ I_{\bbY}\Omega^m_{S/K}}(z)
= \sum_{i=1}^s\binom{n+1}{m} \binom{m_i+n-2}{n}> 0$
since $m_i \ge 2$ for some $i \in \{1, \dots, s\}$.
Hence we get the stated lower bound for the Hilbert polynomial
of~$\Omega^m_{R_{\bbW}/K}$.
In particular, we have $\HP_{\Omega^m_{R_{\bbW}/K}}(z) > 0$.

Furthermore, Proposition~\ref{S3.Prop.02} shows that
$\HP_{\Omega^m_{R_{\bbW}/K}}(z)$ is a constant polynomial.
Now we find an upper bound for $\HP_{\Omega^m_{R_{\bbW}/K}}(z)$.
Clearly, the $R_{\bbW}$-module $\Omega^m_{R_{\bbW}/K}$
is generated by the set
$\{\, dx_{i_1}\wedge\cdots\wedge dx_{i_m}
\mid 0\le i_1< \cdots< i_m \le n \,\}$
consisting of $\binom{n+1}{m}$ elements. This implies
$\HF_{\Omega^m_{R_{\bbW}/K}}(i)\le\binom{n+1}{m}\HF_{\bbW}(i-m)$
for all $i\ge 0$. Hence we get
$\HP_{\Omega^m_{R_{\bbW}/K}}(z)
\le \binom{n+1}{m} \sum_{i=1}^s\binom{m_i+n-1}{n},$
which completes the proof.
\end{proof}

Our next corollary is an immediate consequence
of Propositions~\ref{S6.Prop.01} and~\ref{S6.Prop.03}.

\begin{cor} \label{S6.Cor.04}
Let $\bbW=m_1P_1+\cdots+m_sP_s$ be a fat point scheme
in~$\bbP^n$, and let $m_{\max}=\max\{m_1,\dots,m_s\}$.
The following conditions are equivalent.
\begin{enumerate}
  \item The scheme $\bbW$ is not reduced, i.e., $m_{\max} >1$.

  \item There exists $m \in \{2,\dots,n+1\}$ such that
  $\HP_{\Omega^m_{R_{\bbW}/K}}(z) > 0$.

  \item $\HP_{\Omega^{n+1}_{R_{\bbW}/K}}(z) > 0$.
\end{enumerate}
\end{cor} 

%% file: Sec7.tex
\section{The Hilbert Polynomial of
$\Omega^{n+1}_{R_{\bbW}/K}$ for a Fat Point Scheme $\bbW$}

Given a fat point scheme $\bbW=m_1P_1+\cdots+m_sP_s$
in $\bbP^n$, the Hilbert function of
the $R_{\bbW}$-module $\Omega^1_{R_{\bbW}/K}$
satisfies
$\HF_{\Omega^1_{R_{\bbW}/K}}(i)=
(n+1)\HF_{\bbW}(i-1)+ \HF_{\bbW}(i)-\HF_{\bbV}(i)$
for all $i\in\bbZ$, where $\bbV$ is the fattening of~$\bbW$
(see \cite[Corollary 1.9(i)]{KLL}).
Naturally, we still want to give a formula for
the Hilbert function of the module $\Omega^m_{R_{\bbW}/K}$
for $m\ge 2$.
In fact, we can calculate the Hilbert function
of~$\Omega^{n+1}_{R_{\bbW}/K}$ in the following special case.

\begin{prop}\label{S7.Prop.01}
Let $\bbW= m_1P_1+\cdots+m_sP_s$ be a fat point scheme
in $\bbP^n$ supported at a set of points 
$\bbX=\{P_1,\dots,P_s\}$, and let $\bbY$ be
the subscheme $\bbY = (m_1-1)P_1+\cdots+(m_s-1)P_s$
of~$\bbW$. Suppose that $\bbX$ is contained in a hyperplane.
Then we have
$$
\Omega^{n+1}_{R_{\bbW}/K} \cong  R_{\bbY}(-n-1).
$$
In particular, we have
$\HF_{\Omega^{n+1}_{R_{\bbW}/K}}(i) = \HF_{\bbY}(i-n-1)$
for all $i \in \bbZ$.
\end{prop}

\begin{proof}
Assume that $\bbX \subseteq \calZ^+(H)$, where
$0 \ne\! H \!=\! \sum_{i=0}^n a_iX_i \in S$ and
$a_0, \dots, a_n \!\in K$. By letting
$I = \big\langle\, \frac{\partial F}{\partial X_i}
\mid F \in I_{\bbW}, 0 \le i \le n \, \big\rangle$
and by Corollary \ref{S1.Cor.02}, we have
$$
\Omega^{n+1}_{R_{\bbW}/K} \cong (S/I)(-n-1).
$$
Write $I_{\bbW}= \wp_1^{m_1}\cap\cdots\cap\wp_s^{m_s}$
where $\wp_j$ is the associated prime ideal of $P_j$
in $S$, and let $F\in I_\bbW\setminus\{0\}$.
Clearly, $\frac{\partial F}{\partial X_i} \in \wp_j^{m_j-1}$
for $i=0,\dots,n$ and $j=1,\dots,s$, and so
$\frac{\partial F}{\partial X_i} \in I_{\bbY}$ for $i=0,\dots,n$.
Consequently, we get $I \subseteq I_{\bbY}$.

Now we prove $I \supseteq I_{\bbY}$.
Suppose for a contradiction that there exists
a homogeneous polynomial $G$ such that
$G \in I_{\bbY}\setminus I$. Then we have $HG \in I_{\bbW}$.
Since $H \ne 0$, we may assume that $a_i \ne 0$ for some
$i \in\{0,\dots,n\}$.
We have $\frac{\partial(HG)}{\partial X_i} =
a_iG + H \frac{\partial G}{\partial X_i} \in I$.
Since $G \notin I$, we deduce
$G_1 := H\frac{\partial G}{\partial X_i}
\in I_{\bbY}\setminus I$.
Futhermore, we continue to have $HG_1 \in I_{\bbW}$,
and so
$$
\tfrac{\partial(HG_1)}{\partial X_i}
= \tfrac{\partial(H^2\tfrac{\partial G}{\partial X_i})}
{\partial X_i}
= H^2  \tfrac{\partial^2 G}{\partial X_i^2}
+ 2a_i H \tfrac{\partial G}{\partial X_i}
= H^2  \tfrac{\partial^2 G}{\partial X_i^2}
+ 2a_i G_1 \in I.
$$
This implies that
$G_2 := H^2\tfrac{\partial^2 G}{\partial X_i^2}
\in I_{\bbY} \setminus I$.
Repeating this process, we eventually get
$H^{\deg(G)} \in I_{\bbY}\setminus I$.
On the other hand, since $G \in I_{\bbY}$,
it follows that
$\deg(G) \ge \max\{m_1-1,\dots,m_s-1\}$.
Thus we have $H^{\deg(G)+1}\in I_{\bbW}$, and consequently
$H^{\deg(G)} = \tfrac{1}{a_i(\deg(G)+1)}
\frac{\partial H^{\deg(G)+1}}{\partial X_i}\in I$,
a contradiction. Therefore we get $I_{\bbY} = I$, and hence
$\Omega^{n+1}_{R_{\bbW}/K} \cong (S/I_{\bbY})(-n-1)$,
as desired.
\end{proof}

Let us apply the preceding proposition to a concrete case.

\begin{exam}
Let $K=\bbQ$, and let $\bbW$ and $\bbY$ be the fat point schemes
$\bbW= 2P_1+3P_2+4P_3+2P_4+P_5+7P_6+5P_7$ and
$\bbY = P_1+2P_2+3P_3+ P_4 +6P_6+4P_7$
in~$\bbP^5$, where
$P_1 = (1:1:1:1:1:\frac{15}{6})$,
$P_2 = (1:2:1:1:1:\frac{17}{6})$,
$P_3 = (1:1:2:1:1:\frac{18}{6})$,
$P_4 = (1:2:3:4:5:\frac{55}{6})$,
$P_5 =(1:2:2:1:1:\frac{20}{6})$,
$P_6 = (1:3:2:1:1:\frac{22}{6})$,
and where $P_7 = (1:0:0:1:1:\frac{10}{6})$.
Then $\bbX = \{P_1, \dots, P_7\}$ is contained in
the hyperplane $\calZ^+(X_0 - 4X_3 + 3X_4)$.
Thus Proposition~\ref{S7.Prop.01} yields 
$\Omega^6_{R_{\bbW}/K} \cong R_{\bbY}(-6)$. Hence
the Hilbert function of~$\Omega^6_{R_{\bbW}/K}$ is given by
$$
\HF_{\Omega^6_{R_{\bbX}/K}}: 0 \ 0 \ 0 \ 0 \ 0 \ 0
\ 1 \ 6 \ 21 \ 56 \ 126 \ 252 \ 306 \ 329 \ 336 \
337 \ 337 \cdots.
$$
\end{exam}

Although no formula for the Hilbert function of
the module of K{\"a}hler differential $(n+1)$-forms
of an equimultiple fat point scheme is known, the following
theorem provides a formula for its Hilbert polynomial.

\begin{thm} \label{S7.Thm.03}
Let $\bbX = \{P_1, \dots, P_s\}\subseteq \bbP^n$
be a set of $s$ distinct $K$-rational points,
and let $\nu \ge 1$. Then we have
$\HP_{\Omega^{n+1}_{R_{(\nu+1)\bbX}/K}}(z)=\HP_{\nu\bbX}(z)$.
\end{thm}

\begin{proof}
Let $I=\langle\, \frac{\partial F}{\partial X_i}
\,\mid\, F \in I_{\bbX}, 0 \le i \le n \,\rangle$.
By Corollary \ref{S1.Cor.02}, we have 
$\Omega^{n+1}_{R_{\bbX}/K} \cong (S/I)(-n-1)$.
So, Proposition~\ref{S6.Prop.01} implies
$\HF_{\Omega^{n+1}_{R_{\bbX}/K}}(i)=\HF_{S/I}(i-n-1) =0$
for $i\gg 0$. Let $\mathfrak{M}$ denote the homogeneous
maximal ideal of~$S$. There exists a number $t_1 \in \bbN$
such that $I_{t_1+i} = \mathfrak{M}_{t_1+i}$
for all $i\in \bbN$.
Moreover, by~\cite[4.2]{HC}, there is $t_2 \in \bbN$
such that $(I_{\nu\bbX})_{t_2+i}=(I_{\bbX}^{\nu})_{t_2+i}$
for all $i\in \bbN$.
Let $t=\max\{t_1, t_2, r_{\nu\bbX}+1\}$, let
$r = \binom{n+t}{n} -s$, let $\{F_1,\dots,F_r\}$
be a $K$-basis of the $K$-vector space $(I_{\bbX})_t$,
and let
$$
J = \big\langle\, \tfrac{\partial F}{\partial X_i}
\,\mid\, F\in I_{(\nu+1)\bbX}, 0\le i\le n \,\big\rangle.
$$
Clearly, we have
$I_{(\nu+1)\bbX}\subseteq J \subseteq I_{\nu\bbX}$.
Since $\Omega^{n+1}_{R_{(\nu+1)\bbX}/K}\cong(S/J)(-n-1)$,
and since $I_{\nu\bbX}$ is generated by elements of
degrees $\le t$ (cf.~\cite[Proposition~1.1]{GM}), 
it suffices to show that $J_i = (I_{\nu\bbX})_i$ for 
some $i \ge t$.

We observe that
$$
\begin{aligned}
I_{t} &= \big(\langle\,
 \tfrac{\partial F}{\partial X_i}
 \mid F\in I_{\bbX}, 0 \le i \le n \,\rangle + I_{\bbX}
 \big)_{t}\\
&= \big\langle
 \{\tfrac{\partial F}{\partial X_i} \mid
 F\!\in\!(I_{\bbX})_{t+1}, 0 \le\! i \!\le n\}
 \cup\{ G\tfrac{\partial H}{\partial X_i}
 \mid G\!\in\!(I_{\bbX})_k, H\!\in\! S_{t+1-k}, 0\le\! i\!\le n\}
 \big\rangle_K\\
&= \big\langle\, \tfrac{\partial F}{\partial X_i}
 \mid F\in (I_{\bbX})_{t+1}, 0 \le i \le n
 \,\big\rangle_K + (I_{\bbX})_t\\
&= \big\langle\, \tfrac{\partial (X_jG)}{\partial X_i}
\mid G\in (I_{\bbX})_{t}, 0 \le i,j \le n
 \,\big\rangle_K + (I_{\bbX})_t\\
&= \big\langle \tfrac{\partial F_j}{\partial X_i}
\mid 0\le i\le n, 1\le j\le r \big\rangle_K \mathfrak{M}_1.
\end{aligned}
$$
Here the last two equalities follow from
Euler's relation and the fact that
$(I_{\bbX})_{t+1} = (I_{\bbX})_t\mathfrak{M}_1 =
\langle F_j \mid 1\le j\le r \rangle_K \mathfrak{M}_1$.
(Note that $I_\bbX$ can be generated in degrees $\le t$.)
Therefore we get equalities
$$
\begin{aligned}
(I_{\nu\bbX})_{(\nu rn+\nu+1)t}
&= (I_{\nu\bbX})_{\nu t}\cdot\mathfrak{M}_{(\nu rn+1)t}\\
&= (I_{\nu\bbX})_{\nu t}\cdot
   (\mathfrak{M}^{\nu rn+1})_{(\nu rn+1)t} \\
&=(I_{\bbX}^{\nu})_{\nu t} \cdot
   (\mathfrak{M}^{\nu rn+1})_{(\nu rn+1)t}\\
&=\underbrace{(I_{\bbX})_t\cdots(I_{\bbX})_t}_{\nu\ \text{times}}
 \cdot\underbrace{\mathfrak{M}_t\cdots
 \mathfrak{M}_t}_{\nu rn+1 \ \text{times}} \\
&=\underbrace{(I_{\bbX})_t\cdots(I_{\bbX})_t}_{\nu\ \text{times}}
 \cdot\underbrace{I_t\cdots I_t}_{\nu rn+1\ \text{times}}\\
&=\langle F_1,\dots,F_r\rangle_K^{\nu}
(\langle \partial F_j/\partial X_i \mid 0\le\! i\!\le n,
1\le\! j\!\le r \rangle_K \mathfrak{M}_1)^{\nu rn+1}.
\end{aligned}
$$
Thus we only need to prove the inclusion
$$
\langle F_1,\dots,F_r\rangle_K^{\nu}\cdot(\langle
\partial F_j/\partial X_i \mid 0\le i\le n,
1\le j\le r \rangle_K\mathfrak{M}_1)^{\nu rn+1}\subseteq J.
$$

To this end, we first prove that
$F_{j_1}^{\nu-k}F_{j_2}\cdots F_{j_{k+1}}
\tfrac{\partial F_{j_1}}{\partial X_{i_1}}
\cdots\tfrac{\partial F_{j_1}}{\partial X_{i_{k+1}}} \in J$
for all $i_1, \dots ,i_{k+1} \in \{0, \dots, n\}$ and
$j_1, \dots, j_{k+1}\in \{1, \dots, r\}$
and $0 \le k \le \nu$.
We proceed by induction on $k$. If $k=0$, for $0 \le i_1 \le n$
and $1 \le j_1 \le r$ we have
$$
(\nu+1)F_{j_1}^{\nu}\tfrac{\partial F_{j_1}}{\partial X_{i_1}}
= \tfrac{\partial F_{j_1}^{\nu+1}}{\partial X_{i_1}}\in J.
$$
If $k=1$, for $i_1, i_2 \in \{0,\dots,n\}$ and
$j_1,j_2 \in \{1,\dots,r\}$, we get
$$
\nu F_{j_1}^{\nu-1}F_{j_2}
\tfrac{\partial F_{j_1}}{\partial X_{i_1}}\cdot
\tfrac{\partial F_{j_1}}{\partial X_{i_2}} =
\tfrac{\partial(F_{j_1}^{\nu}F_{j_2})}{\partial X_{i_1}}
\cdot\tfrac{\partial F_{j_1}}{\partial X_{i_2}}-
F_{j_1}^{\nu}\tfrac{\partial F_{j_1}}{\partial X_{i_2}}
\cdot \tfrac{\partial F_{j_2}}{\partial X_{i_1}} \in J.
$$
Now we assume that $2\le k\le\nu$ and
$F_{j'_1}^{\nu-(k-1)}F_{j'_2}\cdots F_{j'_{k}}
\tfrac{\partial F_{j'_1}}{\partial X_{i'_1}}
\cdots\tfrac{\partial F_{j'_1}}{\partial X_{i'_{k}}} \in J$
for all $i'_1, \dots ,i'_{k}\in \{0, \dots, n\}$ and
$j'_1, \dots, j'_{k}\in \{1, \dots, r\}$.
Let $i_1, \dots ,i_{k+1}\in \{0, \dots, n\}$ and
$j_1, \dots, j_{k+1} \in \{1, \dots, r\}$.
We have
$$
\begin{aligned}
& (\nu-k+1) F_{j_1}^{\nu-k}F_{j_2}\cdots F_{j_{k+1}}
\tfrac{\partial F_{j_1}}{\partial X_{i_1}}
\cdots\tfrac{\partial F_{j_1}}{\partial X_{i_{k+1}}} \\
&= \tfrac{\partial (F_{j_1}^{\nu-k+1}F_{j_2}\cdots
 F_{j_{k+1}})}{\partial X_{i_1}}
 \tfrac{\partial F_{j_1}}{\partial X_{i_2}}
 \cdots\tfrac{\partial F_{j_1}}{\partial X_{i_{k+1}}}
 - F_{j_1}^{\nu-k+1}\tfrac{\partial F_{j_2}
 \cdots F_{j_{k+1}}}{\partial X_{i_1}}
 \tfrac{\partial F_{j_1}}{\partial X_{i_2}}\cdots
 \tfrac{\partial F_{j_1}}{\partial X_{i_{k+1}}}\\
&= \tfrac{\partial (F_{j_1}^{\nu-k+1}F_{j_2}\cdots
 F_{j_{k+1}})}{\partial X_{i_1}}
 \tfrac{\partial F_{j_1}}{\partial X_{i_2}}\cdots
 \tfrac{\partial F_{j_1}}{\partial X_{i_{k+1}}}
 - {\textstyle\sum\limits_{l=2}^{k+1}}
 F_{j_1}^{\nu-(k-1)}\tfrac{F_{j_2}\cdots F_{j_{k+1}}}{F_{j_l}}
 \tfrac{\partial F_{j_l}}{\partial X_{i_1}}
 \tfrac{\partial F_{j_1}}{\partial X_{i_2}}
\cdots\tfrac{\partial F_{j_1}}{\partial X_{i_{k+1}}}.
\end{aligned}
$$
By the inductive hypothesis, we have
$F_{j_1}^{\nu-(k-1)}\tfrac{F_{j_2}\cdots F_{j_{k+1}}}{F_{j_l}}
\tfrac{\partial F_{j_1}}{\partial X_{i_2}}
\cdots\tfrac{\partial F_{j_1}}{\partial X_{i_{k+1}}}
\in J$ for all $l=2,\dots,k+1$.
Hence we get
$F_{j_1}^{\nu-k}F_{j_2}\cdots F_{j_{k+1}}
\tfrac{\partial F_{j_1}}{\partial X_{i_1}}
\cdots\tfrac{\partial F_{j_1}}{\partial X_{i_{k+1}}} \in J$.

Consequently, if $k = \nu$, then we have
$F_{j_1}F_{j_2}\cdots F_{j_{\nu}}
\tfrac{\partial F_{j}}{\partial X_{i_1}}
\cdots\tfrac{\partial F_{j}}{\partial X_{i_{\nu+1}}} \in J$
for all $i_1, \dots ,i_{\nu+1}\in \{0, \dots, n\}$ and
$j, j_1, \dots, j_{\nu} \in \{1, \dots, r\}$.
On the other hand, the $K$-vector space
$\langle \frac{\partial F_j}{\partial X_i}
\mid 0\le i\le n, 1\le j\le r \rangle_K^{\nu r+1}$
has a system of generators consisting of elements
of the form
$$
(\tfrac{\partial F_1}{\partial X_0})^{\alpha_{01}}\cdots
(\tfrac{\partial F_1}{\partial X_n})^{\alpha_{n1}} \cdots
(\tfrac{\partial F_r}{\partial X_0})^{\alpha_{0r}}\cdots
(\tfrac{\partial F_r}{\partial X_n})^{\alpha_{nr}}
$$
where $\alpha_{ij}\in \bbN$ satisfy
$\sum_{j=1}^r\sum_{i=0}^n \alpha_{ij} = \nu r+1$.
Set $\alpha_j := \sum_{i=0}^n \alpha_{ij}$ for $j=1,\dots,r$.
Since $\sum_{j=1}^r \alpha_j = \nu r+1$,
there is an index $j\in \{1,\dots, r\}$ such that
$\alpha_j\ge \nu +1$. Also, we have $\nu rn+1 \ge \nu r +1$.
It follows that any element of the $K$-vector space
$$
\langle F_1,\dots,F_r\rangle_K^{\nu}\cdot
\langle \partial F_j/\partial X_i
\mid 0\le i\le n, 1\le j\le r \rangle_K^{\nu rn+1}
$$
is a sum of elements of the form
$F_{j_1}\cdots F_{j_{\nu}}
\tfrac{\partial F_{j}}{\partial X_{i_1}}
\cdots\frac{\partial F_j}{\partial X_{i_{\nu+1}}}G$,
where $i_1, \dots ,i_{\nu+1}\in \{0, \dots, n\}$, where
$j, j_1, \dots, j_{\nu} \in \{1, \dots, r\}$,
and where $G\in S$ is a homogeneous polynomial of degree
$\deg(G)=\nu(rn-1)(t-1)$.
Therefore we get
$$
\langle F_1,\dots,F_r\rangle_K^{\nu}\cdot
(\langle \partial F_j/\partial X_i
\mid 0\le i\le n, 1\le j\le r
\rangle_K\mathfrak{M}_1)^{\nu rn+1} \subseteq J,
$$
and this completes the proof.
\end{proof}

The following corollary follows immediately
from Theorem~\ref{S7.Thm.03}.

\begin{cor} \label{S7.Cor.04}
Let $\bbX = \{P_1, \dots, P_s\}\subseteq \bbP^n$
be a set of $s$ distinct $K$-rational points, and
let $\nu \ge 1$. Then we have
$\HP_{\Omega^{n+1}_{R_{(\nu+1)\bbX}/K}}(z)
= s\binom{\nu+n-1}{n}$.
\end{cor}

In the last part of this section we study
the module of K{\"a}hler differential 2-forms
of fat point schemes $\bbW$ in~$\bbP^n$.
Let us begin with the following sequence
of graded $R_{\bbW}$-modules.

\begin{prop}\label{S7.Prop.06}
Let $\bbW = m_1P_1+\cdots+m_sP_s$ be a fat point
scheme in~$\bbP^n$, and let
$\bbW^{(i)}$ be the $i$-th fattening of~$\bbW$ 
for $i\ge 1$.
Then the sequence of graded $R_{\bbW}$-modules
\begin{equation}\tag{$\mathcal{C}$}
0\rightarrow I_{\bbW^{(1)}}/I_{\bbW^{(2)}}
\stackrel{\alpha}{\rightarrow}
I_{\bbW}\Omega^1_{S/K}/I_{\bbW^{(1)}} \Omega^1_{S/K}
\stackrel{\beta}{\rightarrow}
\Omega^2_{S/K}/I_{\bbW}\Omega^2_{S/K}
\stackrel{\gamma}{\rightarrow} \Omega^2_{R_{\bbW}/K}
\rightarrow 0
\end{equation}
is a complex, where
$\alpha(F+I_{\bbW^{(2)}})= dF+I_{\bbW^{(1)}}\Omega^1_{S/K}$,
where $\beta(G dX_i + I_{\bbW^{(1)}}\Omega^1_{S/K})
= d(G dX_i) + I_{\bbW}\Omega^2_{S/K}$, and where
$\gamma(H+I_{\bbW}\Omega^2_{S/K})=H+(I_{\bbW}\Omega^2_{S/K}
+ dI_{\bbW}\Omega^1_{S/K})$.
Moreover, the following statements hold true.

\begin{enumerate}
\item The map $\alpha$ is injective.

\item The map $\gamma$ is surjective.

\item We have $\im(\beta)=\Ker(\gamma)$.

\item For all $i\ge 0$, we have
$$
\begin{aligned}
\HF_{\Omega^2_{R_{\bbW}/K}}(i+2)
&\ge \tfrac{n(n+1)}{2}\HF_{\bbW}(i)
     + \HF_{\bbW^{(2)}}(i+2)
     -\HF_{\bbW^{(1)}}(i+2)\\
&\quad -(n+1)(\HF_{\bbW^{(1)}}(i+1)
     - \HF_{\bbW}(i+1)).
\end{aligned}
$$
\end{enumerate}
\end{prop}

\begin{proof}
To prove (a), we note that the arguments in the 
proof of~\cite[Theorem~1.7]{KLL} shown that 
the map $\alpha$ is a homogeneous injection. 
Claims~(b) and (c) are implied by the presentation
$$
\Omega^2_{R_{\bbW}/K} \cong
\Omega^2_{S/K}/(I_{\bbW}\Omega^2_{S/K}
+ dI_{\bbW} \Omega^1_{S/K}).
$$
Since $d\circ d =0$, it follows that $\beta\circ\alpha = 0$.
Therefore the sequence ($\mathcal{C}$) is a complex.
In addition, claim~(d) follows from
the fact that ($\mathcal{C}$) is a complex and~(c).
\end{proof}

Now we look at the special case of
an equimultiple fat point scheme $\bbW=\nu\bbX$ 
in~$\bbP^2$. In this case, by taking the homogeneous
components at a large degree $i$ of the
complex~$(\mathcal{C})$, we have the following
exact sequence of $K$-vector spaces.

\begin{prop}\label{S7.Prop.07}
Let $\bbX = \{P_1, \dots, P_s\} \subseteq \bbP^2$ be
a set of $s$ distinct $K$-rational points, and
let $\nu\ge 1$. For $i \gg 0$, the sequence of
$K$-vector spaces
$$
\begin{aligned}
0\longrightarrow
(I_{(\nu+1)\bbX}/I_{(\nu+2)\bbX})_i
&\stackrel{\alpha}{\longrightarrow}
(I_{\nu\bbX}\Omega^1_{S/K}/I_{(\nu+1)\bbX}\Omega^1_{S/K})_i \\
&\stackrel{\beta}{\longrightarrow}
(\Omega^2_{S/K}/I_{\nu\bbX}\Omega^2_{S/K})_i
\stackrel{\gamma}{\longrightarrow}
(\Omega^2_{R_{\nu\bbX}/K})_i
\longrightarrow 0.
\end{aligned}
$$
is exact. Here the maps $\alpha$, $\beta$ and $\gamma$
are defined as in~Proposition~\ref{S7.Prop.06}.
\end{prop}

\begin{proof}
By Proposition~\ref{S7.Prop.06}, we only need to prove
$\im(\alpha)=\Ker(\beta)$. Hence it suffices to
show that the Hilbert polynomial of~$\Omega^2_{R_{\nu\bbX}/K}$
satisfies
$$
\HP_{\Omega^2_{R_{\nu\bbX}/K}}(z)=
\tfrac{(n+2)(n+1)}{2} \HP_{\nu\bbX}(z)
+ \HP_{(\nu+2)\bbX}(z) - (n+2)\HP_{(\nu+1)\bbX}(z).
$$
In~$\bbP^2$, Proposition~\ref{S1.Prop.02} yields
the exact sequence of graded $R_{\nu\bbX}$-modules
$$
0\longrightarrow \Omega^3_{R_{\nu\bbX}/K}
\longrightarrow \Omega^2_{R_{\nu\bbX}/K}
\longrightarrow \Omega^1_{R_{\nu\bbX}/K}
\longrightarrow \mathfrak{m}_{\nu\bbX}
\longrightarrow 0.
$$
Moreover, we have
$\HP_{\Omega^3_{R_{\nu\bbX}/K}}(z) = \HP_{(\nu-1)\bbX}(z)$
by Theorem~\ref{S7.Thm.03}.
Hence, by applying \cite[Corollary~1.9]{KLL}, we get
$$
\begin{aligned}
\HP_{\Omega^2_{R_{\nu\bbX}/K}}(z)
&= \HP_{\Omega^1_{R_{\nu\bbX}/K}}(z) +
\HP_{\Omega^3_{R_{\nu\bbX}/K}}(z) -\HP_{\nu\bbX}(z)\\
&= ((n+2)\HP_{\nu\bbX}(z) - \HP_{(\nu+1)\bbX}(z))
+ \HP_{(\nu-1) \bbX}(z)  - \HP_{\nu\bbX}(z)\\
&= 3s\binom{\nu+1}{2}-s\binom{\nu+2}{2}+s\binom{\nu}{2}\\
&= \tfrac{1}{2}s(3\nu^2-\nu-2)\\
&= 6s \binom{\nu+1}{2} + s\binom{\nu+3}{2}
- 4s \binom{\nu+2}{2} \\
&=\tfrac{(n+2)(n+1)}{2}\HP_{\nu\bbX}(z) +
\HP_{(\nu+2)\bbX}(z)-(n+2)\HP_{(\nu+1)\bbX}(z),
\end{aligned}
$$
and the claim follows.
\end{proof}

The following formula for the Hilbert polynomial
of~$\Omega^2_{R_{\nu\bbX}/K}$ can be extracted
from the proof of this proposition.

\begin{cor}\label{S7.Cor.08}
Let $\bbX = \{P_1,\dots,P_s\} \subseteq \bbP^2$
be a set of $s$ distinct $K$-rational points, and
let $\nu\ge 1$. Then we have
$\HP_{\Omega^2_{R_{\nu\bbX}/K}}(z)
= \frac{1}{2}(3\nu^2-\nu-2)s$.
\end{cor}

In general, we do not have an explicit formula for the 
Hilbert polynomial of $\Omega^{n+1}_{R_{\bbW}/K}$ for 
a non-reduced fat point scheme $\bbW$ in $\bbP^n$.
However, we propose the following conjecture.

\begin{conjecture}
Let $\bbW$ be the fat point scheme 
$\bbW = m_1P_1+\cdots+m_sP_s$ in $\bbP^n$,
and let $\bbY = (m_1-1)P_1+\cdots+(m_s-1)P_s$.
Then we have
$\HP_{\Omega^{n+1}_{R_{\bbW}/K}}(z) = \HP_{\bbY}(z)$.
\end{conjecture}

The above conjecture holds true for $n=1$
by~Proposition~\ref{S2.Prop.01}, and for $n\ge 2$
such that $\bbW$ is an equimultiple fat point
scheme in $\bbP^n$ by Theorem~\ref{S7.Thm.03}.
The following example provides a further instance 
in which the conjecture holds.

\begin{exam}
Let $\bbX=\{P_1,...,P_{10}\}$ be the set of 10 points
on the twisted cubic curve in $\bbP^3$ given by 
$P_1 = (1:1:1:1)$, $P_2 = (1:-1:1:-1)$,
$P_3 = (1:2:4:8)$, $P_4 = (1:-2:4:-8)$,
$P_5 = (1:3:9:27)$, $P_6 = (1:-3:9:-27)$,
$P_7 = (1:4:16:64)$, $P_8 = (1:-4:16:-64)$,
$P_9 = (1:-5:25:-125)$, and $P_{10} = (1:6:36:216)$.
Let $\bbW = P_1+2P_2+4P_3 +3P_4+4P_5+2P_6+
3P_7+7P_8+5P_9+6P_{10}$, and let
$\bbY= P_2 +3P_3 +2P_4+3P_5+P_6+ 2P_7 + 6P_8 +4P_9+5P_{10}$.
A calculation yields that the Hilbert polynomials
of $\Omega^{n+1}_{R_{\bbW}/K}$ and of $R_{\bbY}$
are the same and equal to 141.
\end{exam} 

%% file: Sec8.tex
\section{Fat Point Schemes Supported on Non-Singular Conics}

In~Proposition~\ref{S2.Prop.01}, we described concretely
the Hilbert function of the module of K\"{a}hler
differential $m$-forms when $\bbW$ is a fat point scheme
in~$\bbP^1$. This result leads us to the following question:

\begin{ques}
Can one compute explicitly the Hilbert function
of the bi-graded $R_{\bbW}$-algebra
$\Omega_{R_{\bbW}/K}$ for a fat point scheme
$\bbW= m_1 P_1 + \cdots + m_sP_s$ in~$\bbP^2$?
\end{ques}

In this section we answer some parts of this question.
In particular, we give concrete formulas for the Hilbert
function of the bi-graded algebra $\Omega_{R_{\bbW}/K}$
if $\bbW$ is an equimultiple fat point scheme in~$\bbP^2$
whose support $\bbX$ lies on a non-singular conic.

In what follows, we let $\mathcal{C} = \mathcal{Z}^+(C)$
be a non-singular conic defined by a quadratic polynomial
$C \in S=K[X_0,X_1,X_2]$, we let $\bbX = \{P_1,\dots,P_s\}$
be a set of $s$ distinct $K$-rational points on~$\mathcal{C}$,
and we let $\bbW = m_1P_1 + \cdots + m_sP_s$ be a fat point
scheme in~$\bbP^2$ supported at~$\bbX$.
Suppose that $0 \le m_1 \le \cdots \le m_s$ and $s \ge 4$.
Then it was shown in~\cite[Proposition~2.2]{Ca} that
the regularity index of~$\bbW$ is
$$
r_{\bbW} =  \max\big\{m_s+m_{s-1}-1,
\big\lfloor{\textstyle\sum\limits_{j=1}^s} m_j/2\big\rfloor
\big\}.
$$
Moreover, the Hilbert function of~$\bbW$ can be effectively
computed from the Hilbert function of a certain subscheme
$\bbY$ of~$\bbW$ (see~\cite[Theorem~3.1]{Ca}).

The Hilbert function of the module of K{\"a}hler differential
$1$-forms $\Omega^1_{R_{\bbW}/K}$ satisfies the following
conditions.

\begin{thm}\label{S8.Thm.01}
In the setting above, let $\mu = \sum_{j=1}^s m_j +s$
and $\varrho = m_s+m_{s-1}$.

\begin{enumerate}
\item If $\mu \ge 2\varrho +4$ then
$$
\HF_{\Omega^1_{R_{\bbW}/K}}\!(i) \!=\!
\begin{cases}
    \sum_{j=1}^s\frac{(m_j+1)(3m_j-2)}{2}
    & \!\!\! \mbox{if} \ i \ge \lfloor\tfrac{\mu}{2}\rfloor,
    \smallskip \\
    3\sum_{j=1}^s\binom{m_j+1}{2}-2i-1
    & \!\!\! \mbox{if} \
    r_{\bbW}+2\le i\!<\!\lfloor\tfrac{\mu}{2}\rfloor,
    \smallskip\\
    4\sum_{j=1}^s\binom{m_j+1}{2}-2i-1-\HF_{\bbW}(i-2)
    &\!\!\! \mbox{if} \ i=r_{\bbW}+1,
    \smallskip\\
    \HF_{\bbW}(i)+3\HF_{\bbW}(i-1)-\HF_{\bbW}(i-2)-2i-1
    & \!\!\! \mbox{if} \ 0\le i \le r_{\bbW}.
\end{cases}
$$

\item If $\mu \le 2\varrho+3$ and
$\bbY :=(m_1+1)P_1+\cdots+(m_{s-2}+1)P_{s-2}
+ m_{s-1}P_{s-1} + m_sP_s$, then
$$
\HF_{\Omega^1_{R_{\bbW}/K}}\!(i) \!=\!
\begin{cases}
    \sum_{j=1}^s\frac{(m_j+1)(3m_j-2)}{2}
    & \!\!\! \mbox{if} \ i\ge \varrho+1,
    \smallskip\\
    4\sum_{j=1}^n\binom{m_j+1}{2} -i-1 -\HF_{\bbY}(i-1)
    & \!\!\! \mbox{if} \ r_{\bbW}+1 \le i\le \varrho,
    \smallskip\\
    \HF_{\bbW}(i)+3\HF_{\bbW}(i-1)-\HF_{\bbY}(i-1)-i-1
    & \!\!\! \mbox{if} \ 0 \le i \le r_{\bbW}.
\end{cases}
$$
\end{enumerate}
\end{thm}

\begin{proof}
Let $\bbV$ be the fat point scheme
$\bbV = (m_1+1)P_1 + \cdots + (m_s+1) P_s$
containing~$\bbW$.
By \cite[Corollary~1.9(i)]{KLL}, we have
$$
\HF_{\Omega^1_{R_{\bbW}/K}}(i)
= 3\HF_{\bbW}(i-1)+\HF_{\bbW}(i)-\HF_{\bbV}(i)
$$
for all $i \in \bbZ$.
Also, we have
$r_{\bbV} = \max\{\varrho+1, \lfloor\frac{\mu}{2}\rfloor\}$.

(a)\quad Consider the case $\mu \ge 2\varrho +4$.
In this case, we have
$r_{\bbV} = \lfloor\frac{\mu}{2}\rfloor$.
Since $s \ge 4$, we get the inequality
$$
r_{\bbW} +1 =
\max\{\varrho-1, \lfloor\tfrac{\mu-s}{2}\rfloor\}+1
< \lfloor \tfrac{\mu}{2} \rfloor
= r_{\bbV}.
$$
So, \cite[Corollary~1.9(iii)]{KLL} yields
$\ri(\Omega^1_{R_{\bbW}/K})
\le r_{\bbV} = \lfloor\frac{\mu}{2}\rfloor$.
Moreover, we see that
$$
\begin{aligned}
\HF_{\Omega^1_{R_{\bbW}/K}}(r_{\bbV}-1)
&= 4 \deg(\bbW)- \HF_{\bbV}(r_{\bbV}-1)
 >4 \deg(\bbW)- \HF_{\bbV}(r_{\bbV} + i)\\
&= 4\deg(\bbW) - \deg(\bbV)
= \HF_{\Omega^1_{R_{\bbW}/K}}(r_{\bbV}+i)
\end{aligned}
$$
for all $i \ge 0$,
and hence $\ri(\Omega^1_{R_{\bbW}/K}) = r_{\bbV}$.
Consequently, we can apply \cite[Theorem~3.1]{Ca}
to work out the Hilbert function
of~$\Omega^1_{R_{\bbW}/K}$ with respect to
the value of the degree $i$ as follows:

\begin{enumerate}
\item[(i)]  For $i \ge r_{\bbV} = \lfloor\frac{\mu}{2}\rfloor$,
    the Hilbert function of~$\Omega^1_{R_{\bbW}/K}$ satisfies
    $$
    \HF_{\Omega^1_{R_{\bbW}/K}}(i)
    = 4{\textstyle \sum\limits_{j=1}^s\binom{m_j+1}{2}
    -\binom{m_j+2}{2}
    =\sum\limits_{j=1}^s \frac{(m_j+1)(3m_j-2)}{2}}.
    $$

\item[(ii)]  Let $r_{\bbW}+2\le i<\lfloor\frac{\mu}{2}\rfloor$.
    Then we have $\HF_{\bbW}(i-2) = \HF_{\bbW}(i-1)
    = \HF_{\bbW}(i)= \deg(\bbW) =\sum_{j=1}^s\binom{m_j+1}{2}$
    and $\HF_{\bbV}(i) =  2i+1 + \HF_{\bbW}(i-2)$.
    It follows that
    $$
    {\textstyle
    \HF_{\Omega^1_{\bbW}/K}(i)
    = 4\sum\limits_{j=1}^s\binom{m_j+1}{2}
      - 2i-1- \HF_{\bbW}(i-2)
    = 3\sum\limits_{j=1}^s\binom{m_j+1}{2}- 2i-1.
    }
    $$

\item[(iii)] In the case $i = r_{\bbW} +1$, we have
    $\HF_{\bbW}(i-1)=\HF_{\bbW}(i)=\sum_{j=1}^s\binom{m_j+1}{2}$
    and $\HF_{\bbV}(i) =  2i+1 + \HF_{\bbW}(i-2)$. Thus
    $$
    {\textstyle
    \HF_{\Omega^1_{\bbW}/K}(i)
    = 4\sum\limits_{j=1}^s\binom{m_j+1}{2}-2i-1-\HF_{\bbW}(i-2).
    }
    $$

\item[(iv)] If $0\le i \le r_{\bbW}$, we have
    $\HF_{\bbV}(i) =  2i+1 + \HF_{\bbW}(i-2)$ and
    $$
    \HF_{\Omega^1_{R_{\bbW}/K}}(i) =\HF_{\bbW}(i)
    + 3\HF_{\bbW}(i-1)-2i-1-\HF_{\bbW}(i-2).
    $$
\end{enumerate}
Altogether, we have proved the formula for
the Hilbert function of~$\Omega^1_{R_{\bbW}/K}$
in the case $\mu \ge 2\varrho+4$.

(b)\quad Next we consider the case $\mu\le 2\varrho+3$.
In this case, the relation between Hilbert functions
of~$\bbV$ and of~$\bbY$ follows from \cite[Theorem~3.1]{Ca}.
Also, we have $r_{\bbV} = \varrho + 1$
and $r_{\bbW} = \varrho - 1 < r_{\bbV}$, and hence
$\ri(\Omega^1_{R_{\bbW}/K}) = r_{\bbV}$.
Therefore a similar argument as in the first case
yields the desired formula for the Hilbert function
of~$\Omega^1_{R_{\bbW}/K}$.
\end{proof}

It is worth noting that \cite[Theorem~3.1]{Ca}
and Theorem~\ref{S8.Thm.01} give us a procedure
for computing the Hilbert function of the module of
K\"{a}hler differential $1$-forms of~$R_{\bbW}/K$
from some suitable fat point schemes.
Moreover, $\HF_{\Omega^1_{R_{\bbW}/K}}$ is completely
determined by $s$ and the multiplicities $m_1,\dots,m_s$.

\begin{rem}
On a non-singular conic $\mathcal{C}$, let $\bbW$
be a complete intersection of type $(2, d)$.
Let $P\in \bbW$, and let $\bbY=\bbW\setminus\{P\}$.
The regularity index of the scheme $\bbY$ is $d-1$.
Using Theorem~\ref{S8.Thm.01} we see that
the Hilbert function of~$\Omega^1_{R_{\bbY}/K}$
is independent of the choice of the point~$P$.
\end{rem}

The following proposition can be used to find out
a connection between the Hilbert functions
of~$\Omega^3_{R_{\bbW}/K}$
(as well as of~$\Omega^2_{R_{\bbW}/K}$)
and of a suitable subscheme of~$\bbW$ if $\bbW$
is an equimultiple fat point scheme, i.e.,
if $m_1 = \cdots = m_s =\nu$.
This proposition follows from~\cite[Proposition~4.3]{Ca}.

\begin{prop}\label{S8.Lem.04}
In the setting of Theorem~\ref{S8.Thm.01},
we define
$$
\bbY= \max\{m_1-1,0\}P_1+\cdots+\max\{m_s-1,0\}P_s
$$
if $\mu-s \ge 2\varrho$ and
$$
\bbY= m_1P_1 +\cdots + m_{s-2}P_{s-2}
+ \max\{m_{s-1}-1,0\}P_{s-1}+\max\{m_s-1,0\}P_s
$$
otherwise. Further, let
$q = \max\{\varrho,\lfloor\tfrac{\mu-s+1}{2}\rfloor\}$,
let $\{G_1,\dots,G_{r}\}$ be a minimal
homogenous system of generators of~$I_{\bbY}$,
and let $L$ be the linear form passing through
$P_{s-1}$ and $P_s$.

\begin{enumerate}
\item If $\mu-s \ge 2\varrho$ and $\mu-s$ is odd,
there exist $F_1, F_2 \in (I_{\bbW})_q$ such that
the set $\{CG_1,\dots, CG_{r}, F_1, F_2\}$
is a minimal homogeneous system of generators
of~$I_{\bbW}$.

\item If $\mu-s \ge 2\varrho$ and $\mu-s$ is even,
there exists $F\in (I_{\bbW})_q$ such that
the set $\{CG_1,\dots,CG_{r}, F\}$ is a minimal
homogeneous system of generators of~$I_{\bbW}$.

\item If $\mu-s < 2\varrho$, there exists
$G \in (I_{\bbW})_q$ such that the set
$\{LG_1, \dots, LG_{r}, G\}$ is a minimal
homogeneous system of generators of~$I_{\bbW}$.
\end{enumerate}
\end{prop}

In particular, if $\bbX=\{P_1,\dots,P_s\}$
is a set of $s$ distinct $K$-rational points on
a non-singular conic $\mathcal{C}$, then,
for every $k\ge 1$, there is a minimal homogeneous system
of generators of $I_{k\bbX}$ of the following form.

\begin{cor}\label{S8.Cor.05}
Let $s\ge 4$, let $\bbX=\{P_1,\dots,P_s\}$ be
a set of $s$ distinct $K$-rational points
on a non-singular conic $\mathcal{C}=\mathcal{Z}^+(C)$,
let $\{C,G_1,\dots,G_t\}$ be a minimal homogeneous
system of generators of $I_{\bbX}$, and let $k\ge 1$.

\begin{enumerate}
\item If $s=2v$ for some $v\in \bbN$, then
there exists a minimal homogeneous system of generators
of~$I_{k\bbX}$ of the form
$$
\{\, C^k, C^{k-1}G_1,\dots,C^{k-1}G_t, C^{k-2}F_{21},
C^{k-3}F_{31},\dots, CF_{(k-1)1}, F_{k1} \,\}
$$
where $F_{j1}\in (I_{j\bbX})_{jv}$ for all $j=2,\dots,k$.

\item If $s=2v+1$ for some $v\in \bbN$ and $k$ is even,
then there exists a minimal homogeneous system of generators
of~$I_{k\bbX}$ of the form
$$
\begin{aligned}
\{\, C^k,C^{k-1}G_1,\dots, &C^{k-1}G_t,C^{k-2}F_{21},
\\ & C^{k-3}F_{31},C^{k-3}F_{32}, \dots,CF_{(k-1)1},
CF_{(k-1)2},F_{k1} \,\}
\end{aligned}
$$
where $F_{jl}\in (I_{j\bbX})_{q_j}$ with
$q_j=\lfloor\frac{j(2v+1)+1}{2}\rfloor$
for every $j=2,\dots,k$ and $l=1, 2$.

\item If $s=2v+1$ for some $v\in \bbN$ and $k$ is odd,
then there exists a minimal homogeneous system of generators
of~$I_{k\bbX}$ of the form
$$
\begin{aligned}
\{\, C^k,C^{k-1}G_1,\dots,&C^{k-1}G_t, C^{k-2}F_{21},
\\ &C^{k-3}F_{31},C^{k-3}F_{32}, \dots,CF_{(k-1)1},
F_{k1}, F_{k2} \,\}
\end{aligned}
$$
where $F_{jl}\in (I_{j\bbX})_{q_j}$ with
$q_j=\lfloor \frac{j(2v+1)+1}{2}\rfloor$
for every $j=2,\dots,k$ and $l=1, 2$.
\end{enumerate}
\end{cor}

\begin{proof}
Since $s\ge 4$, we have
$q_k=\max\{2k,\lfloor\frac{sk+1}{2}\rfloor\}
= \lfloor \frac{sk+1}{2}\rfloor$ for every~$k\ge 1$.
By Proposition~\ref{S8.Lem.04} and by induction on~$k$,
we get the claimed minimal homogeneous system
of generators of the ideal $I_{k\bbX}$.
\end{proof}

Now we present a relation between the Hilbert
functions of the module of K\"{a}hler differential
$3$-forms~$\Omega^3_{R_{\nu\bbX}/K}$ and
of~$S/\mathfrak{M}I_{(\nu-1)\bbX}$, where
$\mathfrak{M} = \langle X_0,\dots,X_n\rangle$
is the homogeneous maximal ideal of~$S$.
Here we make the convention that
$I_{(\nu-1)\bbX} := \langle 1\rangle$ if $\nu =1$.

\begin{thm}\label{S8.Thm.06}
Let $s \ge 4$ and $\nu \ge 1$, and let
$\bbX = \{P_1, \dots, P_s\} \subseteq \bbP^2$
be a set of $s$ distinct $K$-rational points
which lie on a non-singular conic
$\mathcal{C} =\mathcal{Z}^+(C)$. Then we have
$\Omega^3_{R_{\nu\bbX}/K} \cong
(S/\mathfrak{M}I_{(\nu-1)\bbX})(-3)$.
In particular, for all $i \in \bbN$, we have
$$
\HF_{\Omega^3_{R_{\nu\bbX}/K}}(i)
=\HF_{S/\mathfrak{M}I_{(\nu-1)\bbX}}(i-3).
$$
\end{thm}

\begin{proof}
Let $\mathcal{B}_1 = \{C,G_1,\dots,G_t\}$ be a minimal
homogeneous system of generators of~$I_{\bbX}$, and let
$J^{(\nu)} = \langle \frac{\partial F}{\partial X_i}
\mid F \in I_{\nu\bbX},0\le i \le n \rangle$.
Note that $r_\bbX = \lfloor \frac{s}{2}\rfloor$.
According to \cite[Proposition~1.1]{GM}, we may assume that
$2\le \deg(G_j) \le \lfloor \frac{s}{2}\rfloor +1$
for $j=1,\dots,t$.
By Corollary~\ref{S1.Cor.02}, we have
$\Omega^3_{R_{\nu\bbX}/K} \cong (S/J^{(\nu-1)})(-3)$.
Moreover, since $\mathcal{C}$ is a non-singular conic,
we have $\langle \tfrac{\partial C}{\partial X_i}
\mid 0 \le i \le 2 \rangle = \mathfrak{M}$,
and hence $J^{(1)}=\mathfrak{M}$.
Thus it suffices to prove the equality
$J^{(\nu+1)}=\mathfrak{M}I_{\nu\bbX}$ for all $\nu \ge 1$.
For $k\ge 2$, let $\mathcal{B}_k$ be the minimal
homogeneous system of generators of~$I_{k\bbX}$
constructed in Corollary~\ref{S8.Cor.05}.
We see that $\deg(F_{(k+1)1}) \ge \deg(F_{k1})+2$
for all $k\ge 2$ and that $\deg(F_{21}) = s$.

If $s=4$ then $\bbX$ is a complete intersection
of type $(2,2)$, and so we may assume
$\mathcal{B}_1 = \{C,G_1\}$ with $\deg(G_1)=2$.
This implies $\deg(F_{21}) = \deg(G_1) +2$.
In the case $s\ge 5$ we have $\deg(F_{21}) = s \ge
\lfloor \frac{s}{2}\rfloor +3
\ge \max\{\deg(G_j)\mid 1\le j\le t\} +2$.
Hence the inclusion
$J^{(\nu+1)} \subseteq \mathfrak{M}I_{\nu\bbX}$
holds true for all $\nu \ge 1$.

Now we proceed by induction on $\nu$ to prove that
$C^{k-1}\mathfrak{M}I_{\nu\bbX} \subseteq J^{(\nu+k)}$
for all $k\ge 1$. In the following, let $k\ge 1$,
$0\le i \le 2$ and $1\le j\le t$. If $\nu=1$, we have
$$
C^{k} \mathfrak{M}=
C^{k} \langle \tfrac{\partial C}{\partial X_i}
\mid 0 \le i \le 2\rangle
= \langle \tfrac{\partial C^{k+1}}{\partial X_i}
\mid 0 \le i \le 2\rangle
\subseteq J^{(1+k)}.
$$
Since $\tfrac{\partial G_j}{\partial X_i}\in\mathfrak{M}$,
we also have
$$
kC^{k-1}G_j \tfrac{\partial C}{\partial X_i}
= \tfrac{\partial (C^{k} G_j)}{\partial X_i}
- C^{k} \tfrac{\partial G_j}{\partial X_i}
\in J^{(1+k)}.
$$
Hence we get $C^{k-1}\mathfrak{M}I_{\bbX} \subseteq J^{(1+k)}$
for all $k\ge 1$, and the claim holds true for $\nu=1$.
Next we assume that $\nu \ge 2$ and that
$C^{k-1}\mathfrak{M}I_{l\bbX} \subseteq J^{(l+k)}$
for $1\le l\le\nu-1$ and all $k\ge 1$.
We distinguish the following two cases.

\medskip
\noindent \textbf{Case (a)}: \ Suppose that $s$ is even.
Using Corollary~\ref{S8.Cor.05}(a) we write
$$
\mathcal{B}_{\nu} =
\{ C^{\nu},C^{\nu-1}G_1,\dots,C^{\nu-1}G_t,
C^{\nu-2}F_{21},\dots,CF_{\nu-1\,1}, F_{\nu1}\}
$$
where $F_{l1} \in (I_{l\bbX})_{sl/2}$.
Note that $\tfrac{\partial F_{l1}}{\partial X_i}
\in \mathfrak{M}I_{(l-1)\bbX}$ for $2 \le l\le \nu$.
It follows from the inductive hypothesis that
$C^{k+\nu-l}\tfrac{\partial F_{l1}}{\partial X_i}
\in C^{k+\nu-l} \mathfrak{M}I_{(l-1)\bbX} \subseteq J^{(\nu+k)}$.
Thus, for $2 \le l\le \nu$, we have
$$
(k+\nu-l)C^{k-1} (C^{\nu-l}F_{l1}) \tfrac{\partial C}{\partial X_i}
= \tfrac{\partial (C^{k+\nu-l} F_{l1})}{\partial X_i}
- C^{k+\nu-l} \tfrac{\partial F_{l1}}{\partial X_i} \in J^{(\nu+k)}.
$$
As above we have $C^{k-1+\nu}\mathfrak{M} \subseteq J^{(\nu+k)}$
and $C^{k+\nu-2}G_1\mathfrak{M} \subseteq J^{(\nu+k)}$.
Therefore we obtain
$C^{k-1}\mathfrak{M}I_{\nu\bbX} \subseteq J^{(\nu+k)}$
for all $k\ge 1$, as desired.

\medskip
\noindent \textbf{Case (b)}: \ Suppose that $s$ is odd.
In this case the minimal homogeneous system of generators
$\mathcal{B}_{\nu}$ of~$I_{\nu\bbX}$ is given by
$$
\begin{aligned}
\mathcal{B}_\nu =
\{&C^{\nu},C^{\nu-1}G_1,\dots,C^{\nu-1}G_t,C^{\nu-2}F_{21},\\
&C^{\nu-3} F_{31}, C^{\nu-3} F_{32},\dots,
CF_{(2l-1)1}, CF_{(2l-1)2}, F_{2l\:\!1}\}
\end{aligned}
$$
if $\nu=2l$ and
$$
\begin{aligned}
\mathcal{B}_\nu =
\{&C^{\nu},C^{\nu-1}G_1,\dots,C^{\nu-1}G_t,C^{\nu-2}F_{21}, \\
& C^{\nu-3} F_{31}, C^{\nu-3} F_{32},
\dots, CF_{2l\:\!1}, F_{(2l+1)1},F_{(2l+1)2}\}
\end{aligned}
$$
if $\nu = 2l +1$, where
$F_{ju}\in (I_{j\bbX})_{\lfloor\frac{js+1}{2}\rfloor}$
(see Corollary~\ref{S8.Cor.05}(b),(c)).
Thus we can use the same argument as in case (a)
and get $C^{k-1}\mathfrak{M}I_{\nu\bbX} \subseteq J^{(\nu+k)}$
for all $k\ge 1$.

Altogether, we have shown that
$C^{k-1}\mathfrak{M}I_{\nu\bbX} \subseteq J^{(\nu+k)}$
for all $\nu,k \ge 1$.
In particular, if $k=1$, then we have
$\mathfrak{M}I_{\nu\bbX} \subseteq J^{(\nu+1)}$, and
the proof is complete.
\end{proof}

If $\nu =1$, we have $\HF_{\Omega^3_{R_{\bbX}/K}}(i)=0$
for $i\ne 3$ and $\HF_{\Omega^3_{R_{\bbX}/K}}(3)=1$.
If $\nu\ge 2$, this Hilbert function can be
described explicitly as follows.

\begin{prop}\label{S8.Prop.07}
In the setting of Theorem~\ref{S8.Thm.06},
let $\mathcal{B}_1=\{C, G_1, \dots, G_t\}$
be a minimal homogeneous system of generators
of~$I_{\bbX}$ as constructed in Corollary~\ref{S8.Cor.05},
let $d_j=\deg(G_j)$ for $j=1,\dots,t$, and let $\nu \ge 2$.
Suppose that $d_1 \le \dots \le d_t$. Then we have
$$
\HF_{\Omega^3_{R_{\nu\bbX}/K}}(i) =
\begin{cases}
 s\binom{\nu}{2}+h_i+\delta_i
 &  \mbox{if}\  i\ge \lfloor \frac{s(\nu-1)}{2}\rfloor +3,\\
 \HF_{\nu\bbX}(i-1)-2i+1+h_i+\delta_i
 &  \mbox{if}\ 2<i<\lfloor \frac{s(\nu-1)}{2}\rfloor +3,\\
 0 &  \mbox{if}\ \ i\le 2.
\end{cases}
$$
Here we let $h_i=\# \{G\in \mathcal{B}_1 \mid \deg(G)=i+1-2\nu\}$,
and $\delta_i$ is defined as follows.
\begin{enumerate}
\item If $s=4$ then $\delta_i=\nu-2$ if $i=2\nu+1$
and $\delta_i=0$ otherwise.

\item If $s=5$ then
$$
\delta_i =
\begin{cases}
 1 &  \mbox{if}\  i=2\nu+1, \\
 3 &  \mbox{if $\nu$ is odd and}\
      2\nu+3\le i \le \frac{5\nu+1}{2}, \\
 3 &  \mbox{if $\nu$ is even and}\
      2\nu+3\le i < \frac{5\nu+2}{2}, \\
 2 &  \mbox{if $\nu$ is even and}\
      i = \frac{5\nu+2}{2}, \\
 0 &  \mbox{otherwise}.
\end{cases}
$$

\item If $s\ge 6$ then
$$
\delta_i =
\begin{cases}
 1 & \mbox{if $s$ is even and}\
     i=2\nu-2k+\frac{ks}{2}+1, 2\le k\le \nu-1,\\
 1 & \mbox{if $s$ is odd and}\
     i=2\nu+k(s-4)+1,
     1\le k\le\lfloor\frac{\nu-1}{2}\rfloor,\\
 2 & \mbox{if $s$ is odd and}\
     i=2\nu+k(s-4) + \frac{s+1}{2}-1,
     1\le k\le \lfloor\frac{\nu-2}{2}\rfloor, \\
 0 & \mbox{otherwise}.
\end{cases}
$$
\end{enumerate}
\end{prop}

\begin{proof}
By Theorem~\ref{S8.Thm.06}, we have
$$
\begin{aligned}
\HF_{\Omega^3_{R_{\nu\bbX}/K}}(i)
&= \HF_{S/\mathfrak{M}I_{(\nu-1)\bbX}}(i-3)\\
&= \HF_{S}(i-3) -
\HF_{\mathfrak{M}I_{(\nu-1)\bbX}}(i-3) \\
&= \HF_{S}(i-3) -
\dim_K (\mathfrak{M}_1(I_{(\nu-1)\bbX})_{i-4}) \\
&= \HF_{S}(i-3) - (\dim_K (I_{(\nu-1)\bbX})_{i-3}
        - \# (\mathcal{B}_{\nu-1})_{i-3})\\
&= \HF_{(\nu-1)\bbX}(i-3) +
\# (\mathcal{B}_{\nu-1})_{i-3}.
\end{aligned}
$$
for all $i\in \bbZ$. Note that we have
$r_{(\nu-1)\bbX}=\lfloor \frac{s(\nu-1)}{2}\rfloor$.
If $i\ge r_{(\nu-1)\bbX} +3$, we obtain
$\HF_{(\nu-1)\bbX}(i-3)=s\binom{\nu}{2}$.
Otherwise, we have
$\HF_{(\nu-1)\bbX}(i-3)=\HF_{\nu\bbX}(i-1)-2i+1$
by \cite[Theorem~3.1]{Ca}.
For every $i\ge 0$, we set
$$
\delta_i := \# (\mathcal{B}_{\nu-1})_{i-3} -
\# \{G\in \mathcal{B}_1 \mid \deg(G)=i+1-2\nu\}
=\# (\mathcal{B}_{\nu-1})_{i-3} - h_i.
$$
Hence the claimed shape of the Hilbert function
of~$\Omega^3_{R_{\nu\bbX}/K}$ follows immediately.
Now we apply Corollary~\ref{S8.Cor.05} to
compute the values of $\delta_i$.
We look at degrees of the elements in the tuple
$$
\mathcal{A}=(C^{\nu-3}F_{2\,1}, C^{\nu-4}F_{3\,1},
\dots,CF_{(\nu-2)\,1}, F_{(\nu-1)\,1})
$$
and get the tuple of degrees
$$
\mathcal{A}^*= (2(\nu-3)+\lfloor\tfrac{2s+1}{2}\rfloor,
2(\nu-4)+\lfloor \tfrac{3s+1}{2}\rfloor,\dots,
2+\lfloor\tfrac{(\nu-2)s+1}{2}\rfloor,
\lfloor\tfrac{(\nu-1)s+1}{2}\rfloor).
$$
Consider the following cases.

\medskip
\noindent (a)\quad Suppose that $s=4$.
Then $\mathcal{B}_1=\{C,G_1\}$ with $\deg(G_1)=2$
and every element in~$\mathcal{A}^*$ equals $2\nu-2$.
Also, $\# (\mathcal{B}_{\nu-1})_{i-3} >0$ only if
$i-3 = 2\nu-2$. In the case $i=2\nu +1$, we have
$\# (\mathcal{B}_{\nu-1})_{i-3} = \nu$ and $h_i =2$,
and so $\delta_i = \nu-2$.
Clearly, $\delta_i=0$ if $i\ne 2\nu +1$.

\medskip
\noindent (b)\quad Suppose that $s=5$. We write
$\deg(C^{\nu-(k+1)}F_{k1})
= 2\nu-2(k+1)+\lfloor\tfrac{5k+1}{2}\rfloor$
for $k=2,\dots,\nu-1$. It is easy to see that
$\deg(C^{\nu-(k+1)}F_{k1}) = \deg(C^{\nu-(k+2)}F_{k+1\:\!1})$
if $k$ is odd and $\deg(C^{\nu-(k+1)}F_{k1})+1
= \deg(C^{\nu-(k+2)}F_{k+1\:\!1})$ otherwise.
\begin{itemize}
\item[(i)] If $\nu$ is even then we have
$$
\delta_i =
\begin{cases}
 1 &  \mbox{if}\  i=2\nu+2, \\
 3 &  \mbox{if}\  2\nu+3 \le i < \frac{5\nu+2}{2}, \\
 2 &  \mbox{if}\  i=\frac{5\nu+2}{2}, \\
 0 &  \mbox{otherwise}.
\end{cases}
$$

\item[(ii)] If $\nu$ is odd then we have
$$
\delta_i =
\begin{cases}
 1 &  \mbox{if}\  i=2\nu+2, \\
 3 &  \mbox{if}\  2\nu+3 \le i \le \frac{5\nu+1}{2}, \\
 0 &  \mbox{otherwise}.
\end{cases}
$$
\end{itemize}

\medskip
\noindent (c)\quad Suppose that $s\ge 6$.
In this case the sequence of elements in~$\mathcal{A}^*$
is strictly increasing.

\begin{itemize}
\item[(i)] If $s$ is even then
$$
\delta_i =
\begin{cases}
 1 &  \mbox{if}\ i=2\nu-2k+\frac{ks}{2}+1,
      2\le k\le \nu-1,\\
 0 &  \mbox{otherwise}.
\end{cases}
$$

\item[(ii)] If $s$ is odd then
$$
\delta_i =
\begin{cases}
 1 &  \mbox{if}\  i = 2\nu+k(s-4)+1,
      1\le k\le\lfloor\frac{\nu-1}{2}\rfloor,\\
 2 &  \mbox{if}\  i = 2\nu + k(s-4) +\frac{s+1}{2}-1,
      1\le k\le\lfloor\frac{\nu-2}{2}\rfloor,\\
 0 &  \mbox{otherwise}.
\end{cases}
$$
\end{itemize}
Altogether, the claims follow.
\end{proof}

\begin{exam} \label{S8.Exam.8}
Let $K=\bbQ$, and let $\bbX=\{P_1,\dots,P_8\}\subseteq \bbP^2$
be the set of 8 points given by
$P_1=(1:1:0)$, $P_2 = (1:3:0)$,
$P_3 = (1:0:1)$, $P_4 = (1:4:1)$,
$P_5 = (1:0:3)$, $P_6 =(1:1:4)$,
$P_7 = (1:4:3)$, and $P_8=(1:3:4)$.
Then $\bbX$ is contained in the non-singular conic
defined by $C= 3X_0^2-4X_0X_1+X_1^2-4X_0X_2+X_2^2$.
In particular, $\bbX$ is a complete intersection
of type $(2,4)$, and hence the set $\mathcal{B}_1$
constructed in Corollary~\ref{S8.Cor.05} is of the form
$\mathcal{B}_1=\{C,G_1\}$ with $\deg(G_1)=4$.
The Hilbert functions of $\nu\bbX$, $1\le \nu\le 3$,
are given by
\[
\begin{array}{ll}
\HF_\bbX:& 1\ 3\ 5\ 7\ 8\ 8\cdots \\
\HF_{2\bbX}: & 1\ 3\ 6\ 10\ 14\ 18\ 21\
             23\ 24\ 24\cdots\\
\HF_{3\bbX}: & 1\ 3\ 6\ 10\ 15\ 21\ 27\ 33\ 38\
             42\ 45\ 47\ 48\ 48\cdots
\end{array}
\]
Moreover, we have
$\HF_{\Omega^3_{R_{\bbX}/K}}:\ 0\ 0\ 0\ 1\ 0\ 0\cdots$.
Now we apply Proposition~\ref{S8.Prop.07} to
compute the Hilbert functions
of~$\Omega^3_{R_{\nu\bbX}/K}$, where $\nu =2,3$.
Since $s=8$ is even, we have
$$
\delta_i =
\begin{cases}
 1 &  \mbox{if}\ i=2\nu+2k+1,
      2\le k\le \nu-1,\\
 0 &  \mbox{otherwise}.
\end{cases}
$$
If $\nu =2$, then we have $\delta_i=0$ for all $i\ge 1$ and
$$
h_i =
\begin{cases}
 1 &  \mbox{if}\ i=5,7,\\
 0 &  \mbox{otherwise}.
\end{cases}
$$
An application of~Proposition~\ref{S8.Prop.07} yields
$$
\HF_{\Omega^3_{R_{2\bbX}/K}}(i) =
\begin{cases}
 8+h_i+\delta_i
 &  \mbox{if}\  i\ge 7,\\
 \HF_{2\bbX}(i-1)+h_i+\delta_i +1-2i
 &  \mbox{if}\ 2 < i < 7,\\
 0 &  \mbox{if}\ \ i\le 2.
\end{cases}
$$
Hence we get
$\HF_{\Omega^3_{R_{2\bbX}/K}}:\ 0\ 0\ 0\ 1\
3\ 6\ 7\ 9\ 8\ 8\cdots$.

Similarly, for $\nu=3$ we have
$$
\delta_i =
\begin{cases}
 1 &  \mbox{if}\ i=11,\\
 0 &  \mbox{otherwise},
\end{cases}
\qquad \mbox{and}\qquad
h_i =
\begin{cases}
 1 &  \mbox{if}\ i=7,9,\\
 0 &  \mbox{otherwise}.
\end{cases}
$$
Therefore we get
$\HF_{\Omega^3_{R_{3\bbX}/K}}:\ 0\ 0\ 0\ 1\
3\ 6\ 10\ 15\ 18\ 23\ 25\ 24\ 24\cdots$.
\end{exam}

\begin{prop}\label{S8.Prop.09}
Let $s \ge 4$ and $\nu \ge 1$, and let
$\bbX = \{P_1, \dots, P_s\} \subseteq \bbP^2$
be a set of $s$ distinct $K$-rational points
which lie on a non-singular conic
$\mathcal{C} =\mathcal{Z}^+(C)$.
For $i\ge 0$, let $h_i, \delta_i$ be defined
as in Proposition~\ref{S8.Prop.07}.

\begin{enumerate}
\item If $\nu =1$, we have
$$
 \HF_{\Omega^2_{R_{\bbX}/K}}(i) =
 \begin{cases}
 0 &  \!\!\!\mbox{if $i\ge s$},\\
 3\HF_{\bbX}(2)-9
 &  \!\!\!\mbox{if}\ i=3,\\
  3\HF_{\bbX}(i-1)-\HF_{\bbX}(i-2)-2i-1
 &  \!\!\!\mbox{if $i<s$ and $i\ne 3$}.
 \end{cases}
$$

\item For every $\nu\ge 2$, let
$\mu = \lfloor \frac{s\nu}{2}\rfloor+2$
and $t= \lfloor \frac{s(\nu-1)}{2}\rfloor+3$.
Then we have
$$
 \HF_{\Omega^2_{R_{\nu\bbX}/K}}\!(i)\!=\!\!
 \begin{cases}
 \frac{s(3\nu +2)(\nu-1)}{2}+h_i+\delta_i
 \hfill   \mbox{if}\
 i\ge \lfloor \frac{s(\nu+1)}{2}\rfloor,\\
 \frac{sv(3v+1)}{2}+h_i\!+\delta_i\!-2i\!-1
 \hfill   \mbox{if}\
 \mu \!\le i \!< \lfloor \frac{s(\nu+1)}{2}\rfloor,\\
 3\HF_{\nu\bbX}(i-1)\!-\!\HF_{\nu\bbX}(i\!-2)\!
 +s\textstyle{\binom{\nu}{2}}\!+h_i\!+\delta_i\!-2i\!-1
 \hfill \  \mbox{if}\ t\!\le i \!< \mu,\\
 4\HF_{\nu\bbX}(i-1)-\HF_{\nu\bbX}(i-2)-4i +h_i+\delta_i
 \hfill   \mbox{if}\ i<t.
 \end{cases}
$$
\end{enumerate}
\end{prop}

\begin{proof}
Clearly, $\HF_{\Omega^2_{R_{\nu\bbX}/K}}(i)=0$
for $i\le 1$. By Proposition~\ref{S1.Prop.02},
we have an exact sequence of graded
$R_{\nu\bbX}$-modules
$$
0\longrightarrow \Omega^3_{R_{\nu\bbX}/K}
\longrightarrow \Omega^2_{R_{\nu\bbX}/K}
\longrightarrow \Omega^1_{R_{\nu\bbX}/K}
\longrightarrow \mathfrak{m}_{\nu\bbX}
\longrightarrow 0.
$$
For $i\ge 2$, we get
$$
\begin{aligned}
\HF_{\Omega^2_{R_{\nu\bbX}/K}}(i)
&=\HF_{\Omega^1_{R_{\nu\bbX}/K}}(i)
+\HF_{\Omega^3_{R_{\nu\bbX}/K}}(i)
-\HF_{\nu\bbX}(i)\\
&\stackrel{(*)}{=} \HF_{\nu\bbX}(i)
+3\HF_{\nu\bbX}(i-1) -\HF_{(\nu+1)\bbX}(i)
+ \HF_{\Omega^3_{R_{\nu\bbX}/K}}(i)
- \HF_{\nu\bbX}(i)\\
&=3\HF_{\nu\bbX}(i-1)-\HF_{(\nu+1)\bbX}(i)
+ \HF_{\Omega^3_{R_{\nu\bbX}/K}}(i)
\end{aligned}
$$
where $(*)$ follows from \cite[Corollary 1.9(i)]{KLL}.

We consider the case $\nu=1$.
We see that  $\HF_{2\bbX}(3)=10$, so
$\HF_{\Omega^2_{R_{\bbX}/K}}(3)
= 3\HF_{\bbX}(2)-10+1=3\HF_{\bbX}(2)-9$.
For $i<s$ and $i\ne 3$, we have
$\HF_{\Omega^2_{R_{\bbX}/K}}(i)
= 3\HF_{\bbX}(i-1) - (2i+1+\HF_{\bbX}(i-2)).$
For $i\ge s$ we get
$\HF_{\Omega^2_{R_{\bbX}/K}}(i)
= 3\HF_{\bbX}(i-1) -\HF_{2\bbX}(i)=3s -3s
=0$.
Thus claim (a) follows.

Now we suppose $\nu \ge 2$.
By Proposition~\ref{S8.Prop.07} we have
$$
\begin{aligned}
\HF_{\Omega^2_{R_{\nu\bbX}/K}}(i)
&=3s\binom{\nu+1}{2}-s\binom{\nu+2}{2}+
s\binom{\nu}{2}+h_i+\delta_i \\
&=\tfrac{s(3\nu+2) (\nu-1)}{2}+h_i+\delta_i
\end{aligned}
$$
for $i\ge \lfloor \frac{s(\nu+1)}{2}\rfloor$.
If $i<\lfloor \frac{s(\nu+1)}{2}\rfloor$, then
$\HF_{(\nu+1)\bbX}(i) = 2i+1+\HF_{\nu\bbX}(i-2)$
by \cite[Theorem~3.1]{Ca}.
So, for $\mu \le i<\lfloor \frac{s(\nu+1)}{2}\rfloor$,
we have
$$
\begin{aligned}
\HF_{\Omega^2_{R_{\nu\bbX}/K}}(i)
&= 2s\binom{\nu+1}{2}+s\binom{\nu}{2}+h_i+\delta_i-2i-1\\
&= \tfrac{sv(3v+1)}{2}+h_i+\delta_i-2i-1.
\end{aligned}
$$
For $t\le i<\mu$, we get
$$
\HF_{\Omega^2_{R_{\nu\bbX}/K}}(i)
= 3\HF_{\nu\bbX}(i-1)-\HF_{\nu\bbX}(i-2)-2i-1+
s\binom{\nu}{2}+h_i+\delta_i.
$$
For $i<t$, it follows from Proposition~\ref{S8.Prop.07}
again that
$$
\begin{aligned}
\HF_{\Omega^2_{R_{\nu\bbX}/K}}(i)
&= 3\HF_{\nu\bbX}(i-1)-(2i+1+\HF_{\nu\bbX}(i-2))\\
&\quad\ +(\HF_{\nu\bbX}(i-1)-2i+1+h_i+\delta_i) \\
&= 4\HF_{\nu\bbX}(i-1)-\HF_{\nu\bbX}(i-2)-4i
+h_i+\delta_i.
\end{aligned}
$$
Therefore claim (b) is completely proved.
\end{proof}

To end this section, we apply the preceding proposition
to compute the Hilbert function of~$\Omega^2_{R_{\nu\bbX}/K}$
in a concrete case.

\begin{exam}
Let $\bbX$ be the complete intersection given
in Example~\ref{S8.Exam.8}. For $1\le \nu \le 3$,
an application of Proposition~\ref{S8.Prop.09} yields
$$
\begin{array}{ll}
\HF_{\Omega^2_{R_{\bbX}/K}} \ : & 0\ 0\ 3\ 6\ 7\ 6\
3\ 1\ 0\ 0\cdots \\
\HF_{\Omega^2_{R_{2\bbX}/K}}: & 0\ 0\ 3\ 9\ 18\
27\ 34\ 39\ 39\ 38\ 35\ 33\ 32\ 32\cdots \\
\HF_{\Omega^2_{R_{3\bbX}/K}}: & 0\ 0\ 3\ 9\ 18\
30\ 45\ 60\ 73\ 84\ 90\ 95\ 95\ 94\ 91\ 89\ 88\ 88\cdots.
\end{array}
$$
\end{exam}

%% file: Sec9.tex
\section{The K\"{a}hler Differential Algebra of $R_{\bbX}/K[x_0]$}

Given a $0$-dimensional scheme $\bbX\subseteq \bbP^n$
such that no point in its support is contained in the
hyperplane $\mathcal{Z}^+(X_0)$, the element $x_0=X_0+I_\bbX$
is a non-zerodivisor of $R_\bbX$. Hence there is a short exact
sequence of graded $R_{\bbX}$-modules
$$
0\longrightarrow R_{\bbX}dx_0
\longrightarrow \Omega^1_{R_{\bbX}/K}
\longrightarrow \Omega^1_{R_{\bbX}/K[x_0]}
\longrightarrow 0
$$
(see \cite[Proposition~3.24]{Ku2}).
As noted in Section~2, the homogeneous $R_\bbX$-linear map
$\gamma: \Omega^1_{R_{\bbX}/K} \rightarrow R_\bbX$
satisfies $\gamma(dx_i)=x_i$ for all $i=0,\dots,n$.
If $fdx_0=0$ for some homogeneous element $f\in R_\bbX$,
then $\gamma(fdx_0) = fx_0 = 0$, and so $f=0$,
since $x_0$ is a non-zerodivisor of~$R_\bbX$.
Hence we have ${\rm Ann}_{R_\bbX}(dx_0)=\langle 0\rangle$.
Consequently, we obtain
$$
\HF_{\Omega^1_{R_\bbX/K[x_0]}}(i)
= \HF_{\Omega^1_{R_\bbX/K}}(i) - \HF_\bbX(i-1)
$$
for all $i\in\bbZ$ and
$\ri(\Omega^1_{R_\bbX/K[x_0]})\le
\max\{ r_\bbX+1,\ri(\Omega^1_{R_\bbX/K})\}$.

Thus the Hilbert functions of $\Omega^1_{R_\bbX/K}$
and  $\Omega^1_{R_\bbX/K[x_0]}$ are strongly related,
and it is straightforward to transfer our earlier results to results
about $\Omega^m_{R_\bbX/K[x_0]}$.
For instance, for a $0$-dimensional subscheme of $\bbP^1$,
an application of Proposition~\ref{S2.Prop.01} yields
the following property.

\begin{prop} \label{S9.Prop.01}
Let $\bbX \subseteq \bbP^1$ be a $0$-dimensional scheme,
and let $I_{\bbX}=\langle F \rangle$, where
$F = \prod_{i=1}^s(X_1-a_iX_0)^{m_i}$ for some $s$,
$m_1,\dots,m_s\ge 1$, and $a_i \in K$ with $a_i\ne a_j$
for $i\ne j$, and let $\mu = \sum_{i=1}^sm_i$.
Then the Hilbert functions of the modules of K{\"a}hler
differential $m$-forms of~$R_\bbX/K[x_0]$ are given by
$$
\HF_{\Omega^1_{R_{\bbX}/K[x_0]}}: \ 0 \ 1 \ 2 \ 3 \
\cdots \ \mu-2 \ \mu-1 \ \mu-2 \ \mu-3 \ \cdots \
\mu-s \ \mu-s \cdots
$$
and $\HF_{\Omega^2_{R_{\bbX}/K[x_0]}}(i)=0$ for all $i\in\bbZ$.
\end{prop}

In general, for $1\!\le m \le\!n+1$,
the module of K\"{a}hler differential $m$-forms
$\Omega^m_{R_{\bbX}/K[x_0]}$ has a presentation
$\Omega^m_{R_{\bbX}/K[x_0]} \!\cong
\Omega^m_{S/K[x_0]}/\!(I_{\bbX}\Omega^m_{S/K[x_0]}
+ d_{S/K[x_0]}I_{\bbX}\Omega^{m-1}_{S/K[x_0]})$
(cf.~\!\cite[Proposition~4.12]{Ku2}).
Moreover, according to \cite[X.83]{SS}, the above exact
sequence induces an exact sequence of graded
$R_{\bbX}$-modules
$$
0\longrightarrow R_{\bbX}dx_0
\wedge_{R_\bbX} \Omega^{m-1}_{R_{\bbX}/K}
\longrightarrow \Omega^m_{R_{\bbX}/K}
\longrightarrow \Omega^m_{R_{\bbX}/K[x_0]}
\longrightarrow 0.
$$
By using these results and by applying the methods given
in~Propositions~\ref{S3.Prop.02}, \ref{S4.Prop.03},
\ref{S5.Prop.02} and~\ref{S6.Prop.03}, we get the
following properties for the Hilbert function and
regularity index of $\Omega^m_{R_{\bbX}/K[x_0]}$.
We leave the detailed proofs of these properties to
the interested reader.

\begin{prop}\label{S9.Prop.02}
Let $\bbX \subseteq \bbP^n_K$ be a $0$-dimensional scheme,
and let $1\le m \le n+1$.

\begin{enumerate}
\item For $i<m$, we have
$\HF_{\Omega^m_{R_{\bbX}/K[x_0]}}(i)=0$.

\item For $m\le i<\alpha_{\bbX}+m-1$, we have
 $\HF_{\Omega^m_{R_{\bbX}/K[x_0]}}(i)=
 \binom{n}{m}\cdot\binom{n+i-m}{n}$.

\item The Hilbert polynomial
of~$\Omega^m_{R_{\bbX}/K[x_0]}$ is constant.

\item  We have
$\HF_{\Omega^m_{R_{\bbX}/K[x_0]}}(r_{\bbX} + m)\ge
\HF_{\Omega^m_{R_{\bbX}/K[x_0]}}(r_{\bbX} + m + 1)
\ge \cdots$, and if
$\ri(\Omega^m_{R_{\bbX}/K[x_0]}) \ge r_{\bbX} + m$ then
$$
\HF_{\Omega^m_{R_{\bbX}/K[x_0]}}(r_{\bbX} + m)
> \cdots > \HF_{\Omega^m_{R_{\bbX}/K[x_0]}} (\ri(\Omega^m_{R_{\bbX}/K[x_0]})).
$$
\item The regularity indices of~$\Omega^m_{R_{\bbX}/K[x_0]}$
satisfies
$$
\ri(\Omega^m_{R_{\bbX}/K[x_0]})\le
\max\{r_{\bbX}+m, \ri(\Omega^1_{R_{\bbX}/K})+m-1\}.
$$
\end{enumerate}
\end{prop}

Finally, for non-reduced fat point schemes in $\bbP^n$,
the Hilbert function and regularity index
of~$\Omega^m_{R_{\bbW}/K[x_0]}$ have the following properties.

\begin{prop}\label{S9.Prop.03}
Let $\bbW = m_1P_1+\cdots+m_sP_s$ be a fat point scheme
in~$\bbP^n$, and let $\bbV$ be the fattening of~$\bbW$.
Suppose that $m_i \geq 2$ for some
$i \in \{1,\dots,s\}$, and let $1 \le m \le n+1$.

\begin{enumerate}
\item The Hilbert polynomial
of~$\Omega^m_{R_{\bbW}/K[x_0]}$ is bounded by
$$
\qquad {\textstyle \sum\limits_{i=1}^s
\binom{n}{m}\binom{m_i+n-2}{n}
\le \HP_{\Omega^m_{R_{\bbW}/K[x_0]}}(z)
\le \sum\limits_{i=1}^s\binom{n}{m}\binom{m_i+n-1}{n}.}
$$

\item We have
$$
\qquad \quad \ri(\Omega^m_{R_{\bbW}/K[x_0]})
\le \min\{\max \{r_{\bbW}+m, r_{\bbV}+m-1\},
\max\{r_{\bbW} + n,r_{\bbV}+n-1\}\}.
$$

\item If $m_1\!\le\!\cdots\le m_s$ and
$\Supp(\bbW)\!=\{P_1,\!\dots,P_s\}$ is in
general position, then
$$
\begin{aligned}
\ri(\Omega^m_{R_{\bbW}/K[x_0]})
\le \min\big\{\,
&\max\{m_s+m_{s-1}+m,
\lfloor \tfrac{\sum_{j=1}^sm_j+s+n-2}{n}\rfloor+m-1\},\\
&\max\{m_s+m_{s-1}+n,
\lfloor\tfrac{\sum_{j=1}^sm_j+s+n-2}{n}\rfloor+n-1\}
\,\big\}.
\end{aligned}
$$
\end{enumerate}
\end{prop}